\documentclass[11pt,reqno]{amsart}
\usepackage{fullpage}
\usepackage{amsfonts}
\usepackage{amssymb}
\usepackage{times}
\usepackage{graphicx}
\usepackage{tikz-cd}
\usepackage{appendix}
\usepackage{color}
\usepackage{mathrsfs}
\usepackage{geometry}
\newtheorem{thm}{Theorem}[section]
\newtheorem{cor}[thm]{Corollary}
\newtheorem{lem}[thm]{Lemma}
\newtheorem{prop}[thm]{Proposition}

\theoremstyle{definition}
\newtheorem{ex}[thm]{Example}
\newtheorem{sit}[thm]{Situation}
\newtheorem{defn}[thm]{Definition}

\theoremstyle{remark}
\newtheorem{rem}[thm]{Remark}
\begin{document}

\title[Short version of title]{A Generalization of Lifting Non-proper Tropical Intersections}
\author{Xiang He}
\maketitle

\begin{abstract}
Let $X$ and $X'$ be closed subschemes of an algebraic torus $T$ over a non-archimedean field. We prove the rational equivalence as tropical cycles in the sense of \cite[\S 2]{meyer2011intersection} between the tropicalization of the intersection product $ X\cdot X'$ and the stable intersection $\mathrm{trop}(X)\cdot\mathrm{trop}(X')$, when restricted to (the inverse image under the tropicalization map of) a connected component $C$ of $\mathrm{trop}(X)\cap\mathrm{trop}(X')$. This requires possibly passing to a (partial) compactification of $T$ with respect to a suitable fan. We define the compactified stable intersection in a toric tropical variety, and check that this definition is compatible with the intersection product in loc.cit.. As a result we get a numerical equivalence between $\overline X\cdot\overline X'|_{\overline C}$ and $\overline{\mathrm{trop}(X)\cdot\mathrm{trop}(X')|_C}$ via the compactified stable intersection, where the closures are taken inside the compactifications of $T$ and $\mathbb R^n$. In particular, when $X$ and $X'$ have complementary codimensions, this equivalence generalizes \cite[Theorem 6.4]{osserman2011lifting}, in the sense that $X\cap X'$ is allowed to be of positive dimension. Moreover, if $\overline X\cap\overline X'$ has finitely many points which tropicalize to $\overline C$, we prove a similar equation as in \cite[Theorem 6.4]{osserman2011lifting} when the ambient space is a reduced subscheme of $T$ (instead of $T$ itself).

\end{abstract}


\section{Introduction}
Let $K$ be an algebraically closed field with a valuation $\mathrm{val}\colon K\rightarrow \mathbb R$.
Let $T$ be an algebraic torus of dimension $n$ over $K$. There is a tropicalization map $\mathrm{trop}\colon T\rightarrow \mathbb R^n$ defined by taking the valuation of every coordinate. Under this map, the image of a pure dimensional subscheme $X$ of $T$ is a balanced polyhedral complex of the same dimension, which is denoted by $\mathrm{trop}(X)$. It is natural to consider under what conditions does the intersection commute with tropicalization, namely, given subschemes $X,X'\subseteq T$ when do we have $\mathrm{trop}(X\cap X)=\mathrm{trop}(X)\cap\mathrm{trop}(X')$. This reduces to a lifting problem since we always have $\mathrm{trop}(X\cap X')\subseteq \mathrm{trop}(X)\cap\mathrm{trop}(X')$. Results in this direction have been applied in \cite{chan2015theta}, studying the connection between the theta characteristics of a $K_4$-curve and the theta characteristics of its minimal skeleton; and in \cite{cartwright2012connectivity}, showing that the tropicalization of an irreducible subvariety of an algebraic torus is connected through codimension one; and in \cite{cartwright2014lifting}, discussing the lifting of divisors on a chain of loops as the skeleton of a smooth projective curve, etc. 

When $\mathrm{trop}(X)$ intersects $\mathrm{trop}(X')$ properly this problem is studied thoroughly by Osserman and Payne in \cite{osserman2013lifting}. They proved that $\mathrm{trop}(X\cap X')=\mathrm{trop}(X)\cap\mathrm{trop}(X')$, which generalizes a well-known result \cite[Lemma 3.2]{bogart2007computing} concerning the lifting when $\mathrm{trop}(X)$ and $\mathrm{trop}(X')$ intersect transversely. Moreover, they gave a lifting formula for the intersection multiplicity of $\mathrm{trop}(X)\cdot\mathrm{trop}(X')$ along a maximal face of $\mathrm{trop}(X)\cap\mathrm{trop}(X')$, where the ambient space is a closed subscheme of $T$ (instead of $T$ itself). See \cite[\S 5]{osserman2013lifting} for details.

The commutativity does not hold when $\mathrm{trop}(X)\cap\mathrm{trop}(X')$ is nonproper. For example one can take hyperplanes $X=\{x=1\}$ and $X'=\{x=1+a\}$ where $\mathrm{val}(a)>0$, then $X$ and $X'$ have empty intersection and same tropicalizations. However, one can still ask about the connections between the intersection cycles $X\cdot X'$ and $\mathrm{trop}(X)\cdot \mathrm{trop}(X')$. As an example, assume $K$ is nonarchimedean, Morrison \cite{morrison2015tropical} proved that when $X$ and $X'$ are plane curves that intersect properly, the tropicalization of the intersection cycle $\mathrm{trop}(X\cdot X')$ is rationally equivalent to $\mathrm{trop}(X)\cdot\mathrm{trop}(X')$ as divisors on the (possibly degenerated) tropical curve $\mathrm{trop}(X)\cap\mathrm{trop}(X')$. In the higher dimensional case, Osserman and Rabinoff proved in \cite[Theorem 6.4]{osserman2011lifting} that when $X$ and $X'$ are of complementary codimension, the number of points of $X\cap X'$, after a suitable compactification of the torus, that tropicalize to (the closure in the corresponding compactification of $\mathbb R^n$ of) a connected component $C$ of $\mathrm{trop}(X)\cap\mathrm{trop}(X')$ is the same as the number of points in $\mathrm{trop}(X)\cdot\mathrm{trop}(X')$ supported on $C$, where both numbers are assumed to be finite and are counted with multiplicities.

In this paper we generalize the result of \cite{osserman2011lifting} in several directions. First assume $X$ and $X'$ do not necessarily intersect properly, hence the intersection multiplicity is not well defined. Instead of counting points of their intersection, we look at the refined intersection product $\overline X\cdot \overline X'$ on $\overline X\cap\overline X'$ which as a cycle class is represented by a formal sum of points supported on $\overline X\cap\overline X'$ (\cite[\S 8]{fulton1998intersection}), where the closures are taken inside the toric variety $X(\Delta)$ associated to a certain unimodular fan $\Delta$ (hence $X(\Delta)$ is smooth). Restricting to the closure of a component of $\mathrm{trop}(X)\cap\mathrm{trop}(X')$ in the corresponding compactification, denoted by $N_\mathbb R(\Delta)$, of $\mathbb R^n$ we have:

\begin{thm}\label{introduction theorem 1}
Let $X$ and $ X'$ be closed subschemes of $T$ of complementary codimensions, $C$ a connected component of $\mathrm{trop}(X)\cap \mathrm{trop}(X')$. Then there exists a fan $\Delta$ such that the degree of the subset of $\overline X\cdot\overline X'$ that tropicalizes to $\overline C$ is the same as that of $\mathrm{trop}(X)\cdot\mathrm{trop}(X')$ supported on $C$.
\end{thm}

In particular we are allowed to consider self-intersections of subschemes of $T$ (see Example \ref{selfintersection}), which is not mentioned in \cite{osserman2011lifting}. The idea of proof is similar to \cite[Theorem 6.4]{osserman2011lifting}, namely, we show that the intersection cycle $\overline X\cdot\overline X'$ can be approached by the intersection of $\overline X$ and a perturbation of $\overline X'$, when restricted to a neighborhood of $\overline C$. Here by perturbation we mean $t\overline X'$ for some $t\in T$ with $\mathrm{val}(x_i(t))$ small enough. The argument requires passing to nonarchimedean analytic spaces. This case will be discussed in Section 4, see Theorem \ref{infinite intersection}. Moreover, a sufficient condition for the fan $\Delta$ will be given.

Theorem \ref{introduction theorem 1} is easily generalized to multiple intersections (see Theorem \ref{infinite multiple intersection}) which plays an important role in our next approach of generalization, namely testing higher dimensional intersections. Let $i_{\overline C}\colon Z_{\overline C}\rightarrow \overline X\cap\overline X'$ be the inclusion of the union of irreducible components of $\overline X\cap\overline X'$ that tropicalize to $\overline C$; note that this is an open and closed subset inclusion. Assuming $\dim (X)+\dim(X')$ is greater than or equal to $n$, 
we prove that, after restricting to $\overline C$, we have $\mathrm{trop}(\overline X\cdot \overline X')$ rationally equivalent to the closure of $\mathrm{trop}( X)\cdot\mathrm{trop}( X')$ as tropical cycles in $N_\mathbb R(\Delta)$. Specifically, we have:

\begin{thm}\label{introduction theorem 2}
Let $X$ and $ X'$ be closed subschemes of $T$ of pure dimensions $k$ and $l$, and $C$ a connected component of $\mathrm{trop}(X)\cap \mathrm{trop}(X')$. Then there is a family of fans $\Delta$ such that $$[\mathrm{trop}(i^*_{\overline C}(\overline X\cdot \overline X'))]=[\overline{\mathrm{trop}(X)\cdot\mathrm{trop}(X')|_{ C}}]\in A_{k+l-n}(N_\mathbb R(\Delta)).$$

\end{thm}

See Section 2.3 or \cite[\S 2]{meyer2011intersection} for the definitions of tropical cycles and rational equivalence on $N_\mathbb R(\Delta)$. In particular, rational equivalence preserves the degrees of zero cycles. In Section 3 we develop a compactified stable intersection $``\cdot_c"$  on $N_\mathbb R(\Delta)$ of two tropical cycles of certain type which extends the stable intersection on $N_\mathbb R$. We check that this compactified stable intersection is compatible with the intersection product defined in loc.cit., which we denote by ``$\ast$", and as a result we have:

\begin{thm}\label{introduction theorem 2s}
Let $X$ and $ X'$ be as in Theorem \ref{introduction theorem 2}. For a certain family of tropical cycles $F$ in $\mathbb R^n$ we have  
$$\deg(\mathrm{trop}(i^*_{\overline C}(\overline X\cdot \overline X'))\cdot_c \overline F)=\deg(\overline{\mathrm{trop}(X)\cdot\mathrm{trop}(X')|_{ C}}\cdot_c \overline F).$$
\end{thm}

Theorem \ref{introduction theorem 2} will be restated and proved as Theorem \ref{intersect rationally equivalent} and Theorem \ref{introduction theorem 2s} as Corollary \ref{compact stable equality}, where the notations are explained. The main idea of the proof of Theorem \ref{introduction theorem 2} is, knowing that there is a natural isomorphism (induced by the tropicalization map) between the Chow rings of $X(\Delta)$ and $N_\mathbb R(\Delta)$ through Minkovski weights on $\Delta$ (see Lemma \ref{chow ring}), one actually only need to check the equality (between degrees) above. This can be accomplished by using the ``multiple intersection" version of Theorem \ref{introduction theorem 1}, since we can restrict to cycles which is 
a sum of products of the tropicalizations of hypersurfaces in $X(\Delta)$. 

Note that Theorem \ref{introduction theorem 2} generalizes Theorem \ref{introduction theorem 1} except that the fan condition is more restrictive, and it is different from just saying that $\mathrm{trop}([\overline X]\cdot[\overline X'])=[\mathrm{trop}(\overline X)]\ast[\mathrm{trop}(\overline X')]$ (Lemma \ref{chow ring}) even when $\mathrm{trop}(X)\cap \mathrm{trop}(X')$ only has one component, see Remark \ref{selfinter rem} and Example \ref{selfinter ex}.

In a complementary direction, motivated by the work of Osserman and Payne, in Section 6 we generalize \cite[Theorem 6.4]{osserman2011lifting} in the case of an ambient space which is a reduced closed subscheme of $T$:

\begin{thm}\label{introduction theorem 3}
Let $Y$ be a reduced closed subscheme of $T$ and $X$ and $ X'$ be subschemes of $Y$ of complementary codimension, let $C$ be a connected component of $\mathrm{trop}(X)\cap\mathrm{trop}(X')$ which is contained in the relative interior of a maximal face $\iota$ of $\mathrm{trop}(Y)$ of multiplicity one. For a certain family of fans $\Delta$, if there are only finitely many points of $\overline X\cap\overline X'$ that tropicalize to $\overline C$ then we have: 
$$\sum_{x\in Z_{\overline C}} i(x,\overline X\cdot\overline X';\overline Y)=\sum_{u\in C} i(u,\mathrm{trop}(X)\cdot\mathrm{trop}(X');\mathrm{trop}(Y)).$$
\end{thm}

See below for explanations of the notations. In the proof of Theorem \ref{introduction theorem 3} we consider the analytification of $Y$. Assuming $Y$ is of dimension $d$, the key point is that locally at a point that tropicalizes to a simple point in $\iota$ we have that $\overline Y^\mathrm{an}$ is isomorphic to the analytic torus of dimension $d$. Hence one can essentially replace $Y$ by an algebraic torus, and all the informations that ``lie in" $\iota$ will be preserved. Therefore the theorem becomes a corollary of \cite[Theorem 6.4]{osserman2011lifting}. See Theorem \ref{lifting reduced ambient space} for details.

\subsection*{Acknowledgements}
I would like to thank Brian Osserman for introducing this problem and for many helpful suggestions. I would also like to thank Sam Payne for pointing out the key point in Theorem \ref{introduction theorem 3} above, and thank Johannes Rau for making the author aware of Meyer's paper \cite{meyer2011intersection}. Moreover, I would like to thank the referee for carefully reading the manuscript and for the suggestions about revising.

\subsection*{Notations}
In the sequel we fix an algebraically closed non-archimedean field $K$ with nontrivial valuation group $G=\mathrm{val}(K)$. Let $N$ be a lattice of dimension $n$. Let $T_N$ be the algebraic torus of dimension $n$ whose lattice of characters, usually denoted by M, is dual to $N$. Denote $N_\mathbb R=N\otimes \mathbb R$. We always identify $N_\mathbb R$ with $\mathbb R^n$ when $N$ is specified. For a polyhedron (resp. polyhedral complex) $P$ we denote the recession cone (resp. recession fan) of $P$ by $\rho(P)$, and the relative interior by $\mathrm{relint}(P)$. Denote $P(m)$ the set of faces of $P$ of dimension $m$. 

Let $\Sigma$ be a fan in $N_\mathbb R$ and $\tau\in\Sigma$. We denote $\mathrm{Star}_\Sigma(\tau)=\{\overline\sigma\subset \mathbb R^n/\mathbb R \tau|\tau\prec\sigma\in\Sigma \}$ as a fan in $\mathbb R^n/\mathbb R\tau$. We also set
$\mathrm{Star}(\tau,\Sigma)$ a fan in $\mathbb R^n$ whose cones are of the form $\tilde \sigma=\{\lambda(\bold x-\bold y)|\lambda\geq 0,\bold x\in \sigma,\bold y\in \tau\}$ for all $\sigma\in\Sigma$ which contains $\tau$ as a face.

Let $X$ and $X'$ be closed subschemes of $Y$ such that $\dim(X)+\dim(X')=\dim(Y)$. We denote $i(x,X\cdot X';Y)$ the intersection multiplicity of $X$ and $X'$ at an isolated point $x$ of $X\cap X'$ at which $Y$ is smooth. If in addition $Y$ is a closed subscheme of an algebraic torus, we denote $i(u,\mathrm{trop}(X)\cdot\mathrm{trop}(X');\mathrm{trop}(Y))$ the multiplicity of $\mathrm{trop}(X)_u\cdot\mathrm{trop}(X')_u$ (at the origin) in $\mathrm{trop}(Y)_u$  for $u\in \mathrm{trop}(X)\cap\mathrm{trop}(X')$, where $\mathrm{trop}(X)_u$ is the star of $\mathrm{trop}(X)$ at $u$, constructed by translating $\mathrm{trop}(X)$ so that $u$ is at the origin and taking the cones spanned by faces of $\mathrm{trop}(X)$ that contain $u$. We usually omit $\mathrm{trop}(Y)$ in $i(u,\mathrm{trop}(X)\cdot\mathrm{trop}(X');\mathrm{trop}(Y))$ when $Y=T_N$ is a given torus. 

Let $X$ be a scheme of finite type over $K$, and $Z$ is a union of connected components of $X$ such that $Z$ is proper, and $\alpha$ a cycle class on $X$ of dimension zero. We denote $A_k(X)$ the $k$-th Chow group of $X$, denote $\int_Z \alpha$ the degree of the restriction of $\alpha$ on $Z$.

\section{Preliminaries}
In this section we recall some facts which are useful for later arguments. Note that most facts works for more general fields $K$.
\subsection{Refined intersection product} Let $i_X\colon X\rightarrow Y$ be a regular imbedding of codimension $d$ of schemes of finite type over $K$, and $f\colon Y'\rightarrow Y$ be a morphism. For any fiber square: 

$$\begin{tikzcd}  X'\rar{i_{X'}}\dar{} &Y'\dar{}\\ X\rar{i_X} &Y,
\end{tikzcd}$$
according to \cite{fulton1998intersection} there is a well-defined pull back map $i_X^!\colon A_k(Y')\rightarrow A_{k-d}(X')$, called the refined Gysin homomorphism.
Note that if $i_{X'}$ is also a regular embedding of codimension $d$ and we have fiber squares: 
$$\begin{tikzcd}  X''\rar{}\dar{} &Y''\dar{}\\ X'\rar{i_{X'}} &Y'
\end{tikzcd}\ \ \ \mathrm{and}\ \ \ 
\begin{tikzcd}  X''\rar{}\dar{} &Y''\dar{}\\ X\rar{i_X} &Y,
\end{tikzcd}$$
then the excess intersection formula implies that $i_{X'}^!=i_{X}^!\colon A_k(Y'')\rightarrow A_{k-d}(X'')$.

Let $Y$ be a smooth variety. Then the diagonal embedding $\delta\colon Y\rightarrow Y\times Y$ is a regular embedding of codimension $\dim(Y)$. Let $X$ and $X'$ be subvarieties of $Y$. The \textit{refined intersection product} $X\cdot X'$ is then defined as $\delta^!([X\times X'])\in A_*(X\cap X')$ with respect to the following square:
$$\begin{tikzcd}  X'\cap X\rar{}\dar{} &X\times X'\dar{}\\ Y\rar{\delta} &Y\times Y.
\end{tikzcd}$$
If $i_X\colon X \rightarrow Y$ is a regular embedding, then $X\cdot X'=i_X^!([X'])$.

Now consider families of cycle classes. Let $Z$ be an irreducible variety of dimension $m$, let $t\in Z$ be a regular closed point, which implies that the inclusion $i_t\colon t\hookrightarrow Z$ is a regular embedding of codimension $m$. Given a morphism $p\colon Y\rightarrow Z$ and a $(k+m)$-cycle $\alpha$ on $Y$, we get a family of $k$-cycle classes $i_t^!([\alpha])\in A_k(Y_t)$ for all $t\in Z$, where $i_t^!\colon A_{k+m}(Y)\rightarrow A_k(Y_t)$ is the refined Gysin homomorphism defined by the following fibre square.

$$\begin{tikzcd}  Y_t\rar\dar{p_t} &Y\dar{p}\\ t\rar{i_t} &Z.
\end{tikzcd}$$

Note that by the construction of refined Gysin homomorphism, we actually get a cycle class in $A_k(|\alpha|_t)$ where $|\alpha| $ is the support of $\alpha$. 

\begin{lem} \label{family}
Let $Z$ be a non-singular variety of dimension $m$, assume $t\in Z$ is rational over the ground field, and $Y$ is smooth over $Z$ of relative dimension $n$. If $\alpha\in A_{k+m}(Y)$ and $\beta \in A_{l+m}(Y)$ then
 $$i_t^!(\alpha)\cdot i_t^!(\beta)=i_t^!(\alpha\cdot\beta)$$ in $A_{k+l-n}(|\alpha|_t\cap |\beta|_t)$.
\end{lem}
 \begin{proof}
A similar result is proved in \cite[Corollary 10.1]{fulton1998intersection} where both sides of the equation are in $A_{k+l-n}(Y_t)$, the same argument works in $ A_{k+l-n}(|\alpha|_t\cap |\beta|_t)$.
\end{proof}

\subsection{Divisors on toric varieties}Let $X(\Sigma)$ be the toric variety associated to a fan $\Sigma$ in $N_\mathbb R$. The set of $T_N$-invariant divisors on $X(\Sigma)$ has an explicit description by the combinatorial informations of $\Sigma$. In the following we list some useful conclusions for later arguments, for more details see \cite[\S 4,\S 6]{cox2011toric}.
Denote the group of $T_N$-invariant Weil divisors on $X(\Sigma)$ by Div$_{T_N}(X(\Sigma))$.

Any ray $\rho\in \Sigma(1)$ gives a codimension one orbit $\mathcal O(\rho)$ whose closure is a $T_N$-invariant prime divisor on $X(\Sigma)$. We denote this divisor by $D_\rho$. Let $u_\rho$ be the minimal lattice generator of $\rho$. Recall that $M$ is the lattice of characters of $T_N$. We then have:

\begin{prop}\label{toric divisor}
$\mathrm{(1)}$ Divisors of the form $D=\sum_{\rho\in\Sigma(1)}a_\rho D_\rho$ are precisely the divisors that is invariant under the torus action on $X(\Sigma)$. In other words we have:
$$\mathrm{Div}_{T_N}(X(\Sigma))=\displaystyle\bigoplus_{\rho\in\Sigma(1)}\mathbb Z D_\rho\subseteq \mathrm{Div}(X(\Sigma)).$$ 

$\mathrm{(2)}$For $m\in M$, the corresponding character $\chi^m$ is a rational function on $X(\Sigma)$ and the associated divisor is given by:$$\mathrm{div}(\chi^m)=\sum_{\rho\in\Sigma(1)}\langle m,u_\rho\rangle D_\rho.$$

$\mathrm{(3)}$ We have an exact sequence:$$0\rightarrow M\rightarrow \mathrm{Div}_{T_N}(X(\Sigma))\rightarrow \mathrm{Cl}(X(\Sigma))\rightarrow 0$$
where the first map is $m\rightarrow \mathrm{div}(\chi^m)$. In particular every Weil divisor is rationally equivalent to a $T_N$-invariant divisor. 

$\mathrm{(4)}$The space of global sections of $\mathcal O(D)$ is given by $$\Gamma(X(\Sigma),\mathcal O(D))=\displaystyle\bigoplus_{m\in P_D\cap M}K\cdot \chi^m$$ where $P_D=\{m\in M_\mathbb R|\langle m,u_\rho\rangle\geq -a_\rho\mathrm{\ for\ all\ }\rho\in \Sigma(1)\}$. 

$\mathrm{(5)}$ The divisor $D$ is Cartier if and only if for each $\sigma\in \Sigma$ there is a $m_\sigma\in M$ such that $\langle m_\sigma,u_\rho\rangle=-a_\rho$ for all $\rho\in \sigma(1)$.

$\mathrm{(6)}$ Let $D=\sum_{\rho\in\Sigma(1)}a_\rho D_\rho$ be a cartier divisor on $X(\Sigma)$, let $m_\sigma$ be   as in (5). Then $D$ is ample if and only if $\langle m_\sigma,u_\rho\rangle>-a_\rho$ for all $\sigma\in \Sigma(n)$ and $\rho\in \Sigma(1)\backslash \sigma (1)$.

\end{prop}

On the other hand, for a full dimensional lattice polytope $P$, we denote by $X_P$ the toric variety of the normal fan of $P$. Let $P_1$ be the set of full dimensional lattice polytopes and $P_2$ be the set of pairs $(X(\Sigma),D)$ such that $\Sigma$ is a complete fan and $D$ is a torus-invariant ample divisor on $X(\Sigma)$. We then have the following relation:

\begin{thm} \label{correspondence of ampleness}
There is an one-to-one correspondences between $P_1$ and $P_2$ given by maps $P\mapsto (X_P,D_P)$ and $(X(\Sigma),D)\mapsto P_D$ that are inverses to each other. In particular, if $D$ is an ample divisor on $X(\Sigma)$ then $\Sigma$ is the normal fan of $P_D$.

\end{thm}

Here $P_D$ is as in Proposition \ref{toric divisor}(4). We won't use the definition of $D_P$ in the sequel, for more details see \cite[\S 6.1]{cox2011toric}.

\subsection{Tropical varieties and the extended tropicalization map}
Let $\Delta$ be a unimodular fan in $N_\mathbb R$. Associated to $\Delta$ there is a tropical variety $N_\mathbb R(\Delta)=\cup_{\tau\in\Delta}N_\mathbb R(\tau)$ with topology induced from natural gluing, where $N_\mathbb R(\tau)=\mathrm{Hom}_{\mathbb R\geq 0}(\tau^{\vee},\mathbb R\cup \{+\infty\})$. The space $N_\mathbb R(\Delta)$ satisfies the following conditions:

(1) $N_\mathbb R(\Delta)$ contains $N_\mathbb R=N_\mathbb R(\{0\})$ as an open dense subset and the addition on $N_\mathbb R$ extends to an action of $N_\mathbb R$ on $N_\mathbb R(\Delta)$, of which the orbits are of the form $\mathrm{Hom}_{\mathbb R_{\geq 0}}(\tau^\perp,\mathbb R)=N_\mathbb R/\mathbb R\tau$. This gives $N_\mathbb R(\Delta)$ a stratification by affine linear spaces: $N_\mathbb R(\Delta)=\coprod_{\tau\in\Delta} N_\mathbb R/\mathbb R\tau$. We denote the orbits by $O(\tau)$ for convenience.

(2) Let $x\in N_\mathbb R$ be a finite point, and $v\in N_\mathbb R$ be a direction vector. Then $x+\lambda v$ converges for $\lambda\rightarrow +\infty$ to a point $\overline x\in N_\mathbb R/\mathbb R\tau$ precisely when $v\in |\Delta|$, and $\tau$ is the unique face which contains $v$ as an (relative) interior point, in which case $\overline x$ is the image of $x$ under the projection $\pi_\tau\colon \mathbb R^n\rightarrow \mathbb R^n/\mathbb R\tau$.

(3) Let $X(\Delta)$ be the toric variety associated to $\Delta$. The tropicalization map on $T_N$ extends naturally to $\mathrm{trop}\colon X(\tau)\rightarrow N_\mathbb R(\tau)$ for $\tau\in \Delta$ by the formula $\langle u,\mathrm{trop}(\xi)\rangle=\mathrm{val}(x^u(\xi)) $ (note that when $\tau=\{0\}$ we get the original tropicalization map $\mathrm{trop}\colon T_N\rightarrow N_\mathbb R$ which is just taking the valuation of each coordinate). This is compatible with the orbits stratifications of $X(\Delta)$ and $N_\mathbb R(\Delta)$, namely we have $\mathrm{trop}(\mathcal O(\tau))=O(\tau)$ where $\mathcal O(\tau)$ denotes the torus orbit of $X(\Delta)$ that corresponds to $\tau$. In particular we have the extension $\mathrm{trop}\colon X(\Delta)\rightarrow N_\mathbb R(\Delta)$. Note that for a subscheme $X$ of $T_N$ we have $\overline{\mathrm{trop}(X)}=\mathrm{trop}(\overline X)$ where the closures are taken in $N_\mathbb R(\Delta)$ and $X(\Delta)$ respectively.

\begin{ex}
Given $e_1,...,e_n$ a basis of $N$, and $e_1',...,e_n'\in M$ the dual basis. Let $\tau=\mathbb R_{\geq 0}e_1+\cdots+\mathbb R_{\geq 0}e_n$. Then $\tau^\vee=\mathbb R_{\geq 0}e_1'+\cdots+\mathbb R_{\geq 0}e_n'$, hence $X(\tau)=\mathbb A^n$ and $N_\mathbb R(\tau)=(\mathbb R\cup\{+\infty\})^n$, as in the following diagram when $n=2$.
$$ 
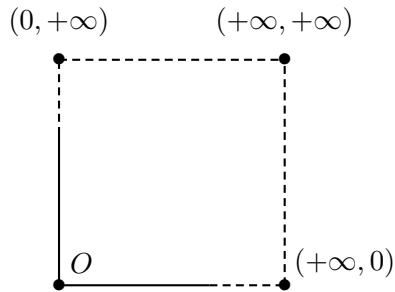
\begin{tikzpicture}[
      scale=1,
      virtual/.style={thick,densely dashed,},
      trans/.style={thick,},
rans/.style={thick,red},
      classical/.style={thin,double,<->,shorten >=4pt,shorten <=4pt,>=stealth}
    ]
 \draw[trans] (0cm,0cm)--(2cm,0cm);
   \draw[trans] (0cm,0cm)--(0cm,2cm);
 \draw[trans](0.3cm,0.3cm)node{$O$};
 \draw[trans](0cm,0cm)node{$\bullet$};
 \draw[virtual] (0cm,2cm)--(0cm,3cm);
 \draw[trans](0cm,3cm)node{$\bullet$};
 \draw[trans](0cm,3.5cm)node{$(0,+\infty)$};
 \draw[trans](3cm,0cm)node{$\bullet$};
 \draw[virtual] (2cm,0cm)--(3cm,0cm);
 \draw[trans](3.8cm,0.3cm)node{$(+\infty,0)$};
 \draw[virtual] (0cm,3cm)--(3cm,3cm);
 \draw[virtual] (3cm,0cm)--(3cm,3cm);
 \draw[trans](3cm,3cm)node{$\bullet $};
 \draw[trans](3cm,3.5cm)node{$(+\infty,+\infty)$};
                 \end{tikzpicture}$$
For $(x_1,...,x_n)\in \mathbb A^n$ we have $\mathrm{trop}(x_1,...,x_n)=(\mathrm{val}(x_1),...,\mathrm{val}(x_n))$, where we set $\mathrm{val}(0)=+\infty$.
\end{ex}

For future reference we also denote $\mathcal V(\tau)$ and $V(\tau)=\mathrm{trop}(\mathcal V(\tau))$ the closed orbits correspond to $\tau$ in $X(\Delta)$ and $N_\mathbb R(\Delta)$ respectively.

\subsection{Miscellaneous}We next list some basic definitions of \cite{osserman2011lifting} which will be frequently used later. 
\begin{defn}\label{compatible fan}
Let $\mathcal P$ be a finite collection of polyhedra in $N_\mathbb R$ and let $\Delta$ be a pointed fan:

(1) The fan $\Delta$ is said to be \textit{compatible with}\footnote{Note that the compatibility is not symmetric. Fix a fan $\Delta$, we usually say ``$P$ is compatible with $\Delta$" but mean that ``$\Delta$ is compatible with $P$", this should be clear when $P$ is not necessarily a fan.} $\mathcal P$ provided that, for all $P\in\mathcal P$ and all cones $\sigma\in\Delta$, either $\sigma\subset \rho(P)$ or $\mathrm{relint}(\sigma)\cap \rho(P)=\emptyset$.

(2) The fan $\Delta$ is said to be a \textit{compactifying fan} for $\mathcal P$ provided that, for all $P\in\mathcal P$ the recession cone $\rho(P)$ is a union of cones in $\Delta$.
\end{defn}

\begin{defn}\label{thickening}
Let $P=\cap_{i=1}^r\{v\in N_\mathbb R| \langle v,u_i \rangle\leq a_i\}$, where $u_i\in M$ and $a_i\in G$, be an integral $G$-affine polyhedron in $N_\mathbb R$. A \textit{thickening} of $P$ is a polyhedron of the form 
$P'= \cap_{i=1}^r\{v\in N_\mathbb R| \langle v,u_i\rangle\leq a_i+\epsilon\}$
for some $\epsilon >0$ in $G$. If $\mathcal P$ is a finite collection of integral $G$-affine polyhedra, a \textit{thickening} of $\mathcal P$ is a collection of (integral $G$-affine) polyhedra of the form $\mathcal P'=\{P'|P\in \mathcal P\}$ where $P'$ denotes a thickening of $P$.
\end{defn}
Note that if $P'$ is a thickening of $P$ then $\rho(P)=\rho(P')$.

\begin{defn}\label{delta thickening}
Let $\Delta$ be an integral fan and let $\mathcal P$ be a finite union of integral $G$-affine polyhedra. A \textit{refinement} of $\mathcal P$ is a finite collection of integral $G$-affine polyhedra $\mathcal P'$ such that every polyhedron of $P'$ is contained in some polyhedron of $\mathcal P$, and every polyhedron of $\mathcal P$ is a union of polyhedra in $\mathcal P'$. A \textit{$\Delta$-decomposition} of $\mathcal P$ is a refinement $\mathcal P'$ of $\mathcal P$ such that $\rho(P)\in\Delta$ for all $P\in\mathcal P'$. A \textit{$\Delta$-thickening} of $\mathcal P$ is a thickening of a $\Delta$-decomposition of $\mathcal P$.
\end{defn}
Note that if $\mathcal P'$ is a $\Delta$-thickening of $\mathcal P$ then $\overline{|\mathcal P|}\subseteq \overline{|\mathcal P'|}^\circ$, where the closures are taken inside $N_\mathbb R(\Delta)$.

\begin{defn}\label{moving lemma}
Let $X$ and $X'$ be closed subschemes of $T_N$, let $C$ be a connected component of $\mathrm{trop}(X)\cap\mathrm{trop}(X')$. 

$\mathrm{(1)}$ A \textit{compacitifying datum} for $X,X'$ and $C$ consists of a pair $(\Delta, \mathcal P)$, where $\mathcal P$ is a finite collection of integral $G$-affine polyhedra in $N_\mathbb R$ such that $\mathrm{trop}(X)\cap\mathrm{trop}(X')\cap|\mathcal P|=|C|$ and $\Delta$ is an integral compactifying fan for $\mathcal P$ which is compatible with $\mathrm{trop}(X')\cap \mathcal P$.

$\mathrm{(2)}$ (\cite[Lemma 4.7]{osserman2011lifting}) If in addition we have $\mathrm{codim}(X)+\mathrm{codim}(X')=n$ then there exists a $\Delta$-thickening $\mathcal P'$ of $\mathcal P$, a number $\epsilon>0$ and a cocharacter $v\in N$ such that: 
$\mathrm{(a)}$ $(\Delta,\mathcal P')$ is a compactifying datum for $X,X'$ and $C$. 
$\mathrm{(b)}$ For all $r\in [-\epsilon,0)\cup(0,\epsilon]$ the set $(\mathrm{trop}(X)+r\cdot v)\cap\mathrm{trop}(X')\cap|\mathcal P'|$ is finite and contained in $|\mathcal P'|^\circ$, and each point lies in the interior of facets of $\mathrm{trop}(X)+r\cdot v$ and $\mathrm{trop}(X')$. We call the tuple $(\mathcal P',\epsilon,v)$ a \textit{tropical moving data}.
\end{defn}

\section{Tropical Intersection Theory}
Let $\Delta$ be a complete unimodular fan. In this section we define the stable intersection of two tropical cycles in the tropical variety $\mathscr X=N_\mathbb R(\Delta)$, where one of them is compatible with $\Delta$ (or equivalently $\Delta$ is a compactifying fan of one of them). We then 
prove that our definition of stable intersection is compatible with the intersection product defined in \cite{meyer2011intersection} (up to rational equivalence).

\subsection{Tropical Intersection Product} 
We first recall some concepts from \cite[\S 2]{meyer2011intersection} about intersection theory on $\mathscr X$. For a treatment of the intersection theory on $N_\mathbb R$ (i.e. $\Delta=\{0\}$) where the intersection product agrees with the stable intersection see \cite{allermann2010first} and \cite{allermann2014rational}.

\begin{defn}
Let $\mathscr X=N_\mathbb R(\Delta)$ be a tropical variety. A \textit{$k$-cycle} on $\mathscr X$ is a collection\footnote{In \cite{meyer2011intersection} a tropical cycle on $\mathscr X$ is defined as a collection of cycles $\alpha_\tau$ on each $O(\tau)$ without taking closures. Here we use the closures $\overline\alpha_\tau$ just to be consistent with our later argument. This will not change the intersection theory on $\mathscr X$.} $\overline\alpha=(\overline\alpha_\tau)_{\tau\in\Delta}$ where each $\overline\alpha_\tau$ is the closure of a (formal sum of) balanced weighted polyhedral complex $\alpha_\tau\subset O(\tau)$ of dimension $k$.

\end{defn}

\begin{defn}
A \textit{tropical rational function} on $N_\mathbb R$ is a continuous piecewise linear function $r\colon N_\mathbb R\rightarrow \mathbb R$ such that there is a finite cover $N_\mathbb R=\cup P_i$ with polyhedra with rational slopes such that $r$ is integral affine on each $P_i$. A \textit{tropical rational function} on a tropical variety $\mathscr X$ is a tropical rational function on its main torus. We denote $\mathrm{Rat}(\mathscr X)$ the set of tropical rational functions on $\mathscr X$.
\end{defn}

\begin{defn}
Let $r$ be a tropical rational function on $\mathscr X$ and $\tau\in\Delta$. We say $r$ \textit{restricts to} $O(\tau)$ if the assignment $z\rightarrow \lim\limits_{x\rightarrow z}r(x)$ defines a tropical rational function on $O(\tau)$, which we denote by $r^\tau(z)$. 
\end{defn}

\begin{defn}\label{cartier divisor}
Let $\overline \alpha$ be a $k$-cycle on $\mathscr X$. A \textit{Cartier divisor} on $\overline \alpha$ is a finite family $\varphi=(U_i,r_i)$ of pairs of open subsets $U_i$ of ${|\overline\alpha|}$ and tropical rational functions $r_i$ on $\mathscr X$ satisfying:

(1) The union of all $U_i$ covers ${|\overline\alpha|}$.
 
(2) For every component $\overline \alpha_\tau$ of $\overline \alpha$ with $\alpha_\tau\subset O(\tau)$ and every chart $U_i$ such that $U_i\cap |\overline \alpha_\tau|\neq\emptyset$ the function $r_i$ must restrict to $O(\tau)$. 

(3) For every component $\overline \alpha_\tau$ of $\overline \alpha$ with $\alpha_\tau\subset O(\tau)$ and all charts $U_i$ and $U_j$ such that $U_i\cap U_j \cap|\overline\alpha_\tau|\neq\emptyset$ there is an affine linear tropical rational function $d$ on $O(\tau)$ such that $r_i^\tau-r_j^\tau=d$ on $U_i\cap U_j\cap|\alpha_\tau|$ and $d$ extends to a continuous function on $U_i\cap U_j$.
\end{defn}

Note that a tropical rational function is a Cartier divisor on $\mathscr X$.

\begin{lem}\label{infinite cell}(\cite[Lemma 2.46]{meyer2011intersection}) Let $P$ be a rational polyhedron and $\tau\in\Delta$ such that $P'=\overline P\cap O(\tau)$ is non-empty and $\dim P=\dim P'+1$. Then there exists a unique primitive lattice vector $v_{P/P'}$ such that\footnote{In \cite[Lemma 2.46]{meyer2011intersection} the vector $v_{P/P'}$ is defined as the unique lattice vector such that $v_{P/P'}\in N\cap\rho(P)\cap\tau$. However, according to \cite[Definition 2.47]{meyer2011intersection} and \cite[Example 2.49]{meyer2011intersection} it is more natural to require $-v_{P/P'}\in N\cap\rho(P)\cap\tau$.} $-v_{P/P'}\in N\cap\rho(P)\cap\tau$. 
\end{lem}

We call $P'$ an infinite cell of $P$ in the case of Lemma \ref{infinite cell}.

\begin{defn}\label{intersect cartier divisor}
Let $\varphi$ be a Cartier divisor on a $(k+1)$-cycle $\overline\alpha_\tau$ with $\alpha_\tau\subset O(\tau)$ as defined in Definition \ref{cartier divisor}. We construct the \textit{intersection product} $\varphi\cdot \overline \alpha_\tau$ as follows: Choose a polyhedral structure on $\alpha_\tau$ such that $\varphi$ is linear on every cell of $\alpha_\tau$. For each cell $P$ choose rational functions $r_P$ in open charts $U_P$ containing $P$. Let $r'_P$ denote the linear part of the restriction of $r_P$ to $P$. We first get a component $\overline E_{\tau}$ with $E_\tau\subset O(\tau)$ whose cells are the codimension one cells $Q$ of $\alpha_\tau$ with weight 
$$w(Q)=\sum_{Q\subsetneq P\in\alpha_\tau}w(P)r'_P(v_{P/Q})-r'_Q(\sum_{Q\subsetneq P\in\alpha_\tau}w(P)v_{P/Q})$$
where $v_{P/Q}$ is a primitive lattice generator of $\mathrm{Star}_P(Q)$ and $w(P)$ is the weight of $P$ in $\alpha_\tau$. For each $\sigma\in\Delta$ with $\tau\subsetneq \sigma$ we get a component $\overline E_{\sigma}$ where $E_\sigma \subset O(\sigma)$ consists of infinite cells $P'=\overline P\cap O(\sigma)$ of $P$ for all $P\subset\alpha_\tau$ such that $\dim(\overline P\cap O(\sigma))=\dim(\alpha_\tau)-1$ with weight:
$$w(P')=w(P)[N(\sigma)_{P'}:N_P(\sigma)]r'_{P}(v_{P/P'}).$$
Here $N(\sigma)_{P'}=\pi_\sigma(N)\cap \mathrm{span}(P')$ and $N_P(\sigma)=\pi_\sigma(\pi_\tau(N)\cap \mathrm{span}(P))$.
Then $\varphi\cdot \overline \alpha_\tau$ is defined as $\varphi\cdot \overline \alpha_\tau=\sum_{\tau\subset\sigma}\overline E_{\sigma}$. This is well defined according to \cite[Theorem 2.48]{meyer2011intersection}. We extend this definition by linearality to the set of all $(k+1)$-cycles $\overline\alpha$ on $\mathscr X$.
\end{defn}

\begin{ex} \label{tropical inter ex}
Let $\Delta$ be the complete fan in $N_\mathbb R=\mathbb R^2$ whose facet $\sigma_i$ is the $i$-th quadrant for $i=1,2,3,4$. Let $\rho_x,\rho_x',\rho_y,\rho_y'$ be the rays of $\Delta$ as below (on the right). We have $\mathscr X=(\mathbb R\cup\{-\infty,+\infty\})^2$.
$$ 
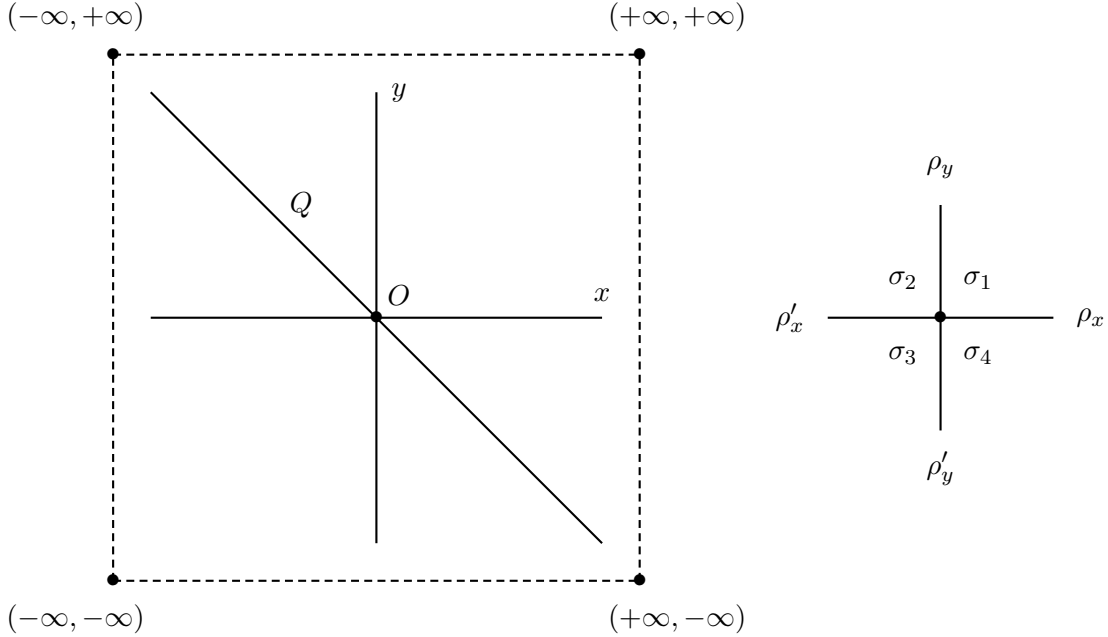
\begin{tikzpicture}[
      scale=1, 
      virtual/.style={thick,densely dashed,},
      trans/.style={thick,},
rans/.style={thick,red},
      classical/.style={thin,double,<->,shorten >=4pt,shorten <=4pt,>=stealth}
    ]
 \draw[trans] (-3cm,0cm)--(3cm,0cm);
   \draw[trans] (0cm,-3cm)--(0cm,3cm);
 \draw[trans](0.3cm,0.3cm)node{$O$};
 \draw[trans](0cm,0cm)node{$\bullet$};
 \draw[virtual] (3.5cm,-3.5cm)--(3.5cm,3.5cm);
 \draw[virtual] (-3.5cm,-3.5cm)--(-3.5cm,3.5cm);
 \draw[virtual] (-3.5cm,3.5cm)--(3.5cm,3.5cm);
 \draw[virtual] (-3.5cm,-3.5cm)--(3.5cm,-3.5cm);
 \draw[trans](3.5cm,3.5cm)node{$\bullet$};
 \draw[trans](3.5cm,-3.5cm)node{$\bullet$};
 \draw[trans](-3.5cm,3.5cm)node{$\bullet$};
 \draw[trans](-3.5cm,-3.5cm)node{$\bullet$};
 \draw[trans](4cm,4cm)node{$(+\infty,+\infty)$};
 \draw[trans](-4cm,-4cm)node{$(-\infty,-\infty)$};
 \draw[trans](4cm,-4cm)node{$(+\infty,-\infty)$};
 \draw[trans](-4cm,4cm)node{$(-\infty,+\infty)$};
   \draw[trans] (3cm,-3cm)--(-3cm,3cm);
 \draw[trans](3cm,0.3cm)node{$x$};
 \draw[trans](0.3cm,3cm)node{$y$};
 \draw[trans](-1cm,1.5cm)node{$Q$};

 \draw[trans] (6cm,0cm)--(9cm,0cm); 
 \draw[trans] (7.5cm,-1.5cm)--(7.5cm,1.5cm);
 \draw[trans](8cm,0.5cm)node{$\sigma_1$};
 \draw[trans](7cm,0.5cm)node{$\sigma_2$};
 \draw[trans](7cm,-0.5cm)node{$\sigma_3$};
 \draw[trans](8cm,-0.5cm)node{$\sigma_4$};
 \draw[trans](9.5cm,0cm)node{$\rho_x$};
 \draw[trans](5.5cm,0cm)node{$\rho_{x}'$};
 \draw[trans](7.5cm,2cm)node{$\rho_y$};
 \draw[trans](7.5cm,-2cm)node{$\rho_y'$};
 \draw[trans](7.5cm,0cm)node{$\bullet$};

                 \end{tikzpicture}$$

Let $\varphi$ be the tropical rational function on $\mathscr X$ such that $\varphi(x,y)=0$ if $x+y\leq 0$ and $\varphi(x,y)=x+y$ if $x+y\geq 0$. We calculate $\varphi\cdot \mathscr X=\varphi\cdot\overline \alpha_\tau$ where $\alpha_\tau =N_\mathbb R$ using symbols in Definition \ref{intersect cartier divisor}. Take the polyhedral structure on $\alpha_\tau$ which contains two facets $P_+=\{(x,y)|x+y\geq 0\}$ and $P_-=\{(x,y)|x+y\leq 0\}$ and a codimension one cell $Q=\{(x,y)|x+y= 0\}$. We have $U_{P_+}=U_{P_-}=N_\mathbb R(\Delta)$ and $r_{P_+}=r_{P_-}=\varphi$, and $r_{P_+}'(x,y)=x+y$ and $r_{P_-}'(x,y)=0$. Restricting to $Q$ we have $r_{P_+}'=r_{P_-}'=r_Q'=0$.

We have $E_{\tau}=Q$. To calculate the weight of $Q$ note that $w(P_+)=w(P_-)=1$, and we can pick $v_{P_+/Q}=(0,1)$ and $v_{P_-/Q}=(0,-1)$. It follows that $w(Q)=r_{P_+}'(v_{P_+/Q})+r_{P_-}'(v_{P_-/Q})=1$.

There are four infinite cells to consider: $P_x=O(\rho_x)$ and $P_y=O(\rho_y)$ are contained in the closure of $P_+$ while $P_x'=O(\rho_x')$ and $P_y'=O(\rho_{y}')$ are contained in the closure of $P_-$. In order to calculate $w(P_x)$ and $w(P_y)$, we first take $v_{P_+/P_x}=(-1,0)$ and $v_{P_+/P_y}=(0,-1)$. One also checks that $N(\rho_x)_{P_x}=\pi_{\rho_x}(N)$ and $N_{P_+}(\rho_x)=\pi_{\rho_x}(N\cap \mathrm{span}(P_+))=\pi_{\rho_x}(N)=N(\rho_x)_{P_x}$, and similarly $N_{P_+}(\rho_y)=N(\rho_y)_{P_y}$. Thus we have $w(P_x)=r'_{P_+}((-1,0))=-1$ and $w(P_y)=r'_{P_+}((0,-1))=-1$. Similarly, taking $v_{P_-/P_x'}=(1,0)$ and $v_{P_-/P_y'}=(0,1)$, we have $w(P_x')=w(P_y')=0$.

It follows that $E_{\rho_x}=-O(\rho_x)$ and $E_{\rho_y}=-O(\rho_y)$ and $E_{\rho_x'}=0$ and $E_{\rho_y'}=0$. Hence $\varphi\cdot \mathscr X=\overline Q-V(\rho_x)-V(\rho_y).$

\end{ex}

\begin{rem}\label{intersection in rn} 
In  \cite{allermann2010first} a similar notion of the intersection product between a tropcical rational function on $O(\tau)$ (up to translation by a linear function on $O(\tau)$) and a cycle in $O(\tau)$ is defined, which we denote by ``$\cdot_\tau$". In Definition \ref{intersect cartier divisor}, if $\varphi$ is a Cartier divisor on $V(\tau)$, then it induces a tropical rational function $\varphi_\tau$ on $O(\tau)$ (again up to translation by a linear function). We then have $E_{\tau}=\varphi_\tau\cdot_\tau \alpha_\tau$, which is also equal to the stable intersection of $\varphi_\tau\cdot_\tau O(\tau)$ and $\alpha_\tau$ inside $O(\tau)$. 
\end{rem}

\begin{defn}
Let $\mathscr X=N_\mathbb R(\Delta)$ be a tropical variety. Let $Z_k(\mathscr X)$ be the group of $k$-cycles on $\mathscr X$. We define subgroups:
$$R_k(\mathscr X)=\mathrm{span}_\mathbb Z\{r\cdot C| r\in\mathrm{Rat}(\mathscr X), C\in Z_{k+1}(\mathscr X)\}$$
$$R_k'(\mathscr X)=\mathrm{span}_\mathbb Z\{f_*(C)|f\colon \mathscr Y\rightarrow \mathscr X \mathrm{\ toric\ morphism,\ }C\in R_k(\mathscr Y)\}.$$
Then the \textit{$k$-th Chow group} of $\mathscr X$ is $A_k(\mathscr X)=Z_k(\mathscr X)/R_k'(\mathscr X)$. A tropical $k$-cycle $\overline \alpha$ is rationally equivalent to $\overline \beta $, denoted by $[\overline \alpha]=[\overline \beta]$, if $\overline \alpha-\overline \beta\in R_k'(\mathscr X)$.
\end{defn}
Here $f_*$ is the pushforward of tropical cycles, see \cite[Definition 2.51]{meyer2011intersection} for its definition. According to \cite[Remark 2.68]{meyer2011intersection} every cycle on $\mathscr X$ is rationally equivalent to a sum of products of Cartier divisors (with $\mathscr X$). This defines an intersection product:
$$``\ast"\colon A_{n-k}(\mathscr X)\times A_{n-l}(\mathscr X)\rightarrow A_{n-k-l}(\mathscr X)$$
which makes $A_*(\mathscr X)$ a graded ring. More precisely, if $[\overline \alpha]=[\varphi_1\cdots\varphi_k\cdot \mathscr X] \in A_{n-k}(\mathscr X)$ and $[\overline \beta]\in A_{n-l}(\mathscr X)$, then $[\overline \alpha]*[\overline\beta]=[\varphi_1\cdots\varphi_k\cdot \overline \beta]\in A_{n-k-l}(\mathscr X)$.

Note that this intersection product of two cycles $\overline\alpha$ and $\overline \beta$ with $\alpha,\beta\subset N_\mathbb R$ does not necessarily coincide with the closure of the stable intersection $\alpha\cdot\beta$, except that $\Delta$ is compatible with one of $\alpha$ and $\beta$ (Lemma\ref{intersection compatible}). The main reason is that, unlike stable intersection, passing to the product with a cartier divisor may create cycles supported on the boundary. See the example below.

\begin{ex}
Consider the same $\Delta$ and same figure as in Example \ref{tropical inter ex}. Let $\alpha=Q\subset N_\mathbb R$ be the line defined by $x+y=0$ with multiplicity one. We have $[\overline \alpha]=[V(\rho_x)]+[V(\rho_y)]$ by loc.cit.. Obviously the closure of the stable intersection of $\alpha$ and itself is empty. We now calculate $[\overline \alpha] *[\overline \alpha]$. Consider the Cartier divisor $\varphi=(U_i,r_i)_{1\leq i\leq 4}$ where $U_i=\cup_{\tau\subset\sigma_i}O(\tau)=N_\mathbb R(\sigma_i)$. Let $r_1(x,y)=-x-y$ and $r_2(x,y)=-y$ and $r_3(x,y)=0$ and $r_4(x,y)=-x$. Straightforward calculation shows that $\varphi\cdot \mathscr X= V(\rho_x)+V(\rho_y)$, hence $[\overline \alpha] *[\overline\alpha]=[\varphi\cdot\overline  \alpha]=[(-\infty,+\infty)+(+\infty,-\infty)].$

\end{ex}



In addition, every cycle on $\mathscr X$ is rationally equivalent to a formal sum of closed $N_\mathbb R$-orbits (\cite[Theorem 2.59]{meyer2011intersection}) as well as a formal sum of closures of subfans of $\Delta$ (\cite[Lemma 2.66]{meyer2011intersection}). As a result we have:

\begin{thm}(\cite[Corollary 2.67]{meyer2011intersection})\label{chow group}
We have group isomorphisms:
$$A_*(X(\Delta))\xrightarrow{\varphi}MW_*(\Delta)\xrightarrow{\phi}A_*(N_\mathbb R(\Delta)).$$ 
 
\end{thm}

Here $MW_*(\Delta)$ is the ring of Minkowski weights on $\Delta$, see \cite{fulton1994intersection} for the construction of $\varphi$. See also \cite[\S 2]{osserman2013lifting} for an illustration of $\phi$. Note that $\phi$ is actually a ring isomorphism (Lemma \ref{intersection compatible}) and that the composition $\phi\circ\varphi$ is indeed (at least when $X(\Delta)$ is projective) induced by the tropicalization map, see Lemma \ref{chow ring}.

\subsection{Compactified stable intersection with certain tropical cycles}
In general it is not obvious to define the stable intersection for two arbitrary cycles in $\mathscr X$. A basic issue is that taking limits of perturbed intersections does not work in the compactified case. For example any translation of a line in $\mathbb T\mathbb P^3$, the tropical projective space of dimension 3, which passes through $(+\infty,+\infty,+\infty)$  still passes through the same point,
but we expect them to have empty intersection. 

However,  we are able to define the compactified stable intersection  for a restricted class of cycles, namely cycles that are compatible\footnote{If $\Delta$ is not complete we may restrict to cycles of which $\Delta$ is a compactifying fan.} with $\Delta$, hence have the same codimension in every affine strata $O(\tau)\subset \mathscr X$. Before we state the definition some balancing and compatibility conditions need to be checked:
\begin{lem}\label{projection polyhedral complex}
Let $\Sigma$ be a polyhedral complex in $\mathbb R^n$ which is compatible with $\Delta$ and let $\tau\in\Delta$. Then one can choose polyhedral structures on $\Sigma$ and $\overline \Sigma\cap O(\tau)$ such that there is a one-to-one correspondence induced by $\pi_\tau$ between faces of $\Sigma$ whose recession cones contain $\tau$ and faces of $\overline \Sigma\cap O(\tau)$.
\end{lem}
\begin{proof}
First we take a refinement of $\Sigma$ such that for any $\sigma\in\Sigma$ we have $\rho(\sigma)\in\Delta$. This is possible since $\Delta$ is compatible with $\Sigma$, and we can refine the structure on $\Sigma$ such that $\sigma\cap\delta$ is a face of $\Sigma$ for all $\sigma\in\Sigma$ and $\delta\in \Delta$. It follows from Lemma 3.9 of \cite{osserman2011lifting} that $$\overline\Sigma\cap O(\tau)=\displaystyle\bigcup_{\sigma\in\Sigma,\tau\subset\rho(\sigma)} \pi_\tau(\sigma).$$

Take $\sigma\in\Sigma$ a face such that $\tau\subset\rho(\sigma)$, then we have $\sigma=\displaystyle\bigcap_{u_i\in I}\{ x\in\mathbb R^n|\langle  x,u_i\rangle\geq b_i\}$ where $I\subset \tau^\vee$ is a finite set.
It follows that $\sigma+\mathbb R\tau=\displaystyle\bigcap_{u_i\in I\cap\tau^{\perp}}\{ x\in\mathbb R^n|\langle  x,u_i\rangle\geq b_i\}$ and $$\pi_\tau(\sigma)=\pi_\tau(\sigma+\mathbb R\tau)=\displaystyle\bigcap_{u_i\in I\cap\tau^\perp}\{ y\in\mathbb R^n/\mathbb R\tau|\langle  y,u_i\rangle\geq b_i\}$$
This is a polyhedron in $\mathbb R^n/\mathbb R\tau$ and will be considered as a face of $\pi_\tau(\Sigma)$. 
We need to show that $\pi_\tau(\sigma)$ are compatible for all $\sigma\in \Sigma$ such that $\tau\subset\rho(\sigma)$.

Note that faces of $\pi_\tau(\sigma)$ are of the form $$F_J=\displaystyle\bigcap_{u_i\in (I\cap\tau^\perp)\backslash J}\{ y\in\mathbb R^n/\mathbb R\tau|\langle  y,u_i\rangle\geq b_i\}\bigcap\displaystyle\bigcap_{u_i\in  J}\{ y\in\mathbb R^n/\mathbb R\tau|\langle  y,u_i\rangle= b_i\}$$
where $J$ is a subset of $I\cap\tau^\perp$. The preimage of $F_J$ in $\sigma$ is 
$$E_J=\displaystyle\bigcap_{u_i\in I\backslash J}\{ x\in\mathbb R^n|\langle  x,u_i\rangle\geq b_i\}\bigcap\displaystyle\bigcap_{u_i\in  J}\{ x\in\mathbb R^n|\langle  x,u_i\rangle= b_i\}.$$
This is an one-to-one correspondence from faces of $\pi_\tau(\sigma)$ to faces of $\sigma$ whose recession cone contains $\tau$.

To complete the proof notice that if $\sigma_1,\sigma_2\in \Sigma$ are two faces then $\pi_\tau(\sigma_1)\cap\pi_\tau(\sigma_2)=\pi(\sigma_1\cap\sigma_2)$ since their recession cones all contain $\tau$. It follows that $\tau\subset \rho(\sigma_1)\cap\rho(\sigma_2)=\rho(\sigma_1\cap\sigma_2)$ and hence $\pi_\tau(\sigma_1\cap\sigma_2)$ is a face of both $\pi_\tau(\sigma_1)$ and $\pi_\tau(\sigma_2)$ due to the argument above.
\end{proof}

Note that the correspondence in Lemma \ref{projection polyhedral complex} preserves codimension of every face, in particular we get an induced weight for every maximal face of $\overline \Sigma\cap O(\tau)$. The following proposition follows directly:

\begin{prop}\label{boundary tropical cycle}
Let $\Sigma\subset \mathbb R^n$ be a tropical cycle which is compatible with $\Delta$, let $\tau\in\Delta$ be a  face. Then $\overline \Sigma\cap O(\tau)$ is a tropical cycle in $O(\tau)$ with the induced structure from $\Sigma$.
\end{prop}

We are now able to define the compactified stable intersections in this special situation:
\begin{defn}\label{compactified stable intersection}
Let $\Sigma$ be a tropical cycle in $\mathbb R^n$ such that $\Delta$ is compatible with $\Sigma$. Let $\gamma\subset O(\tau)$ for some $\tau\in\Delta$. We define the \textit{compactified stable intersection} of $ \overline\Sigma$ and $\overline\gamma$ and denote by $\overline \gamma\cdot_c \overline  \Sigma$ to be the closure of the stable intersection of $\overline\Sigma\cap O(\tau)$ and $\gamma$ as tropical cycles on $O(\tau)$.
\end{defn}

Intuitively the stable intersection is a right action of the multiplicative group of tropical cycles compatible with $\Delta$ on the group of usual tropical cycles. Next we check the associativity of this action.

\begin{prop}\label{associativity of compactified stable intersection}
Let $\Sigma_1,\Sigma_2$ be tropical cycles in $\mathbb R^n$ compatible with $\Delta$ and of codimensions $l$ and $m$ respectively. Let $\tau\in\Delta$ be a face of dimension $d$. Then we have 
$$\overline{\Sigma_1\cdot\Sigma_2}\cap O(\tau)=(\overline \Sigma_1\cap O(\tau))\cdot (\overline\Sigma_2\cap O(\tau))$$ as tropical cycles in $O(\tau)$. In particular for any tropical cycle $\overline\gamma\subset \mathscr X$ we have the asociativity:
$$(\overline \gamma\cdot_c\overline \Sigma_1)\cdot_c\overline\Sigma_2=\overline\gamma\cdot_c(\overline\Sigma_1\cdot_c\overline\Sigma_2).$$
 
\end{prop}
\begin{proof}
Take refinement of $\Sigma_i$ such that for $\sigma_i\in\Sigma_i$ we have $\rho(\sigma_i)\in\Delta$ and $\sigma_1\cap\sigma_2$ is a face of both $\Sigma_1$ and $\Sigma_2$ (first refine $\Sigma_1$ as in Lemma \ref{projection polyhedral complex} and then take intersections of faces of $\Sigma_1$ and $\Sigma_2$). The induced structure on $\overline \Sigma_i\cap O(\tau)$ also satisfies that $\delta_1\cap\delta_2$ is a face of $\overline \Sigma_i\cap O(\tau)$ for all $\delta_i\in\overline \Sigma_i\cap O(\tau)$ and $i=1,2.$

It is easy to check that, with the induced structure, the correspondence in Lemma \ref{projection polyhedral complex} between $\Sigma_1\cap\Sigma_2$ and $(\overline\Sigma_1\cap O(\tau))\cap(\overline\Sigma_2\cap O(\tau))$ is still valid. We then have (set-theoretically):
\begin{equation*}
\begin{split}
\overline{\Sigma_1\cdot\Sigma_2}\cap O(\tau)&=\displaystyle\bigcup_{\sigma_i\in\Sigma_i\mathrm{\ and\ }\dim(\sigma_1+\sigma_2)=n}\overline{\sigma_1\cap\sigma_2}\cap O(\tau)\\&=\displaystyle\bigcup_{\begin{tiny}\begin{array}{cc}\sigma_i\in\Sigma_i\mathrm{\ and\ }\tau\subset\rho(\sigma_i)\\\dim(\sigma_1+\sigma_2)=n\end{array}\end{tiny}}\pi_\tau(\sigma_1\cap\sigma_2)\\&=\displaystyle\bigcup_{\begin{tiny}\begin{array}{cc}\sigma_i\in\Sigma_i\mathrm{\ and\ }\tau\subset\rho(\sigma_i)\\\dim(\sigma_1+\sigma_2)=n\end{array}\end{tiny}}\pi_\tau(\sigma_1)\cap\pi_\tau(\sigma_2)
\\&=\displaystyle\bigcup_{\begin{tiny}\begin{array}{cc}\delta_i\in(\Sigma_i\cap O(\tau))\\\dim(\delta_1+\delta_2)=n-d\end{array}\end{tiny}}\delta_1\cap\delta_2\\&=(\overline \Sigma_1\cap O(\tau))\cdot (\overline\Sigma_2\cap O(\tau)).
\end{split}
\end{equation*}

On the other hand, for any $\delta\in(\overline\Sigma_1\cap O(\tau))\cap(\overline\Sigma_2\cap O(\tau))$ of codimension $l+m$ in $O(\tau)$ and $\sigma\in\Sigma_1\cap\Sigma_2$ with $\pi_\tau(\sigma)=\delta$, we have Star$(\sigma,\Sigma_i)=\mathrm{Star}(\delta,\overline\Sigma_i\cap O(\tau))\times \mathbb R\tau$ and hence
$$i(\sigma, \Sigma_1\cdot\Sigma_2)=i(\delta,(\overline\Sigma_1\cap O(\tau))\cdot(\overline\Sigma_2\cap O(\tau)))$$

\end{proof}
 

Given a tropical cycle $\Sigma\subset N_\mathbb R$ such that $\sigma\in\Sigma$ implies $\rho(\sigma)\in\Delta$, according to \cite[Definition 5.1]{allermann2014rational}, there is a natural balanced structure on $\rho(\Sigma)$ where the weight of $\tau\in\rho(\Sigma)$ is defined by:
$$m(\tau)=\sum_{\sigma\in\Sigma,\rho(\sigma)=\tau}m(\sigma).$$

\begin{lem}\label{boundary of recession fan}
If $\Delta$ is compatible with $\Sigma$, then $\rho(\overline{\Sigma}\cap O(\tau))=\overline{\rho(\Sigma)}\cap O(\tau)$ as tropical cycles on $O(\tau)$.
\end{lem}
\begin{proof}
Let $\sigma\in\Sigma$, note that if $\tau\subset\rho(\sigma)$ then $\pi_\tau(\rho(\sigma))=\rho(\pi_\tau(\sigma))$. The lemma follows directly from Lemma \ref{projection polyhedral complex} and the definition of structures on the recession fans.
\end{proof}

It turns out that, in the compactified sense, intersecting with $\Sigma$ is numerically the same as intersecting with $\rho(\Sigma)$ when the two tropical cycles have complementary codimension:

\begin{prop}\label{intersect recession}
Let $\Sigma$ and $\gamma$ be as in Definition \ref{compactified stable intersection}. Assume $\dim(\Sigma)+\dim(\gamma)=n$. We have $\deg(\overline \gamma\cdot_c\overline \Sigma)=\deg(\overline \gamma\cdot_c{\overline{\rho(\Sigma)}})$.
\end{prop}
\begin{proof}
Follows from Lemma \ref{boundary of recession fan} and \cite[Theorem 5.3]{allermann2014rational}
\end{proof}

\subsection*{} We are now able to check the compatibility of the compactified stable intersection ``$\cdot_c$" and the intersection product ``$\ast$":

\begin{lem}\label{intersection compatible}
If $\beta\subset N_\mathbb R$ is a tropical cycle which is compatible with $\Delta$ and $\alpha\subseteq O(\tau)$, then $[\overline \alpha]\ast[\overline \beta]=[\overline \alpha\cdot_c \overline \beta]$.
\end{lem}
\begin{proof}
First assume $|\beta|=|\Delta(n-1)|$. Then we can find a cartier divisor $\varphi$ such that $\varphi\cdot \mathscr X=\overline\beta$. Indeed, assume $[\overline \beta]=\sum_{\rho\in\Delta(1)}a_\rho[V(\rho)]$. Let $f$ be the support function on $\Delta$ such that $f(u_\rho)=a_\rho$ where $u_\rho$ as in \S 2.2. Then $f\cdot \mathscr X=\overline F-\sum a_\rho V(\rho)$ where $F$ is a codimension one subfan of $\Delta$. Hence $[\overline F]=\sum a_\rho[V(\rho)]$, and Theorem \ref{chow group} implies that $\overline \beta=\overline F$ as tropical cycles. For all $\sigma\in\Delta(n)$ take a (unique) $m_\sigma\in M$ such that $f|_\sigma=\langle \ \cdot\ ,m_\sigma\rangle$. Now let $\varphi=(U_\sigma,r_\sigma)_{\sigma\in\Delta(n)}$ where $U_\sigma=\cup_{\tau\prec\sigma}O(\tau)$ and $r_\sigma=f-\langle\ \cdot\ , m_\sigma\rangle$. Note that $r_\sigma$ is zero on $\sigma$. Straightforward calculation shows that $\overline \beta=\varphi\cdot \mathscr X$ (all infinite cells have weight zero) and, using symbols in Remark \ref{intersection in rn}, we have $\overline \beta\cap O(\tau)=\varphi_\tau\cdot_\tau O(\tau)$ as tropical cycles on $O(\tau)$. Now by loc.cit. we have $$[\overline \beta]*[\overline\alpha]=[\varphi\cdot\overline \alpha]=[\sum_{\tau\subset \sigma}\overline E_{\sigma}]=[\overline E_{\tau}]=[\overline{\varphi_\tau\cdot_\tau \alpha}]=[\overline \beta\cdot_c\overline \alpha].$$

The associativity of ``$\ast$" and ``$\cdot_c$" guarantees that the conclusion is true for $\beta$ any cycle of subfans of $\Delta$. In particular the morphism $\phi$ in Theorem \ref{chow group} is a ring isomorphism.

For the general case note that Theorem \ref{chow group} and Proposition \ref{intersect recession} shows that $[\overline \beta]=[\overline{\rho(\beta)}]$. Hence we have:
$$[\overline \alpha]\ast[\overline \beta]=[\overline \alpha]\ast[\rho(\overline\beta)]=[\overline\alpha\cdot_c\overline{\rho(\beta)}]=[\overline \alpha\cdot_c\overline\beta].$$
One can check the last equality by counting the degree of the intersection products of both side with all cycles $\overline \gamma$ where $\gamma$ is a subfan of $\Delta$ of suitable codimension, which, as we already showed, is the same as the degree of $\overline\alpha\cdot_c\overline{\rho(\beta)}\cdot_c\overline\gamma$ and $\overline \alpha\cdot_c\overline\beta\cdot_c\overline\gamma$ respectively.
\end{proof}

\section{The Complementary Codimension Case}
In this section we assume that we are given two closed subschemes $X$ and $X'$ of $T_N$ which are of complementary codimension but possibly have positive dimensional intersection. In this case $X\cdot X'$ is a zero dimensional cycle class whose degree is not well-defined in general. We show that, after a suitable compactification and restriction to a component $C$ of $\mathrm{trop}(X)\cap\mathrm{trop}(X')$, the degree above is invariant under rational equivalence, and is equal to the corresponding intersection number on the tropical side.

Let $(\Delta, \mathcal P)$ be a compactifying datum for $X,X'$ and $C$. Let $(\mathcal P',\epsilon,v)$ be a tropical moving data for $(\Delta,\mathcal P)$. We then have(\cite[Corollary 4.8]{osserman2011lifting}):

\begin{lem}\label{connected component}
In the situation above we have $$\overline{\mathrm{trop}(X)}\cap\overline{\mathrm{trop}(X')}\cap\overline{|\mathcal P'|}=\overline C\subset \overline{|\mathcal P|}\subset\overline{|\mathcal P'|}^\circ$$ and for all $r\in [-\epsilon,0)\cup(0,\epsilon]$ we have
$$\overline{\mathrm{trop}(X)+r\cdot v}\cap \overline{\mathrm{trop}(X')}\cap\overline{|\mathcal P'|}=({\mathrm{trop}(X)+r\cdot v})\cap {\mathrm{trop}(X')}\cap{|\mathcal P'|}\subset |\mathcal P'|^\circ\subset \overline{|\mathcal P'|}^\circ$$where all closures are taken in $N_R(\Delta)$.
\end{lem}

Let $Z_{\overline C}$ be the union of irreducible components of $\overline X\cap\overline X'$ whose tropicalization intersects $\overline C$, where $\overline X$ and $\overline X'$ are closures of $X$ and $X'$ in $X(\Delta)$ respectively. We claim that $Z_{\overline C}$ is a union of connected components of $\overline X\cap \overline X'$, hence has a natural scheme structure which is in particular flat over $\overline X\cap \overline X'$. Indeed, according to Lemma \ref{connected component} we know $\overline C$ is a connected component of $\overline{\mathrm{trop}(X)}\cap \overline{\mathrm{trop}(X')}$, and the claim follows from the fact that the tropicalization of every irreducible component of $\overline X\cap \overline X'$ is a connected set. In particular $\mathrm{trop}(Z_{\overline C})\subseteq \overline C$.

Moreover, \cite[Corollary 4.17]{osserman2011lifting} shows that $Z_{\overline C}$ is proper over the ground field, hence there is a well-defined degree of every cycle class of dimension zero on $Z_{\overline C}$. In the following we take $\mathcal P=C$ with the induced polyhedral complex structure. It follows that $(\Delta,\mathcal P)$ is a compactifying datum if $\Delta$ is a compactifying fan for $C$:

\begin{sit}\label{complementary codimension}
Let $X$ and $X'$ be closed subschemes of $T_N$ of pure dimension $k$ and $l$ respectively, where $k+l=n$. Let $C$ be a connected component of trop$(X)\cap$trop$(X')$. Let $\Delta$ be a unimodular compactifying fan for $C$, let $i_{\overline C}\colon Z_{\overline C}\rightarrow \overline X\cap \overline X'$ be the inclusion.
\end{sit}

\begin{thm}\label{infinite intersection}
In Situation \ref{complementary codimension}, we have $$\int_{Z_{\overline C}}i^*_{\overline C}(\overline X\cdot \overline X')=\sum_{u\in C} i(u,\mathrm{trop}(X)\cdot\mathrm{trop}(X'))$$where $\overline X\cdot \overline X'$ is the refined intersection product.
\end{thm}

\begin{proof}[Proof of Theorem \ref{infinite intersection}] As in \cite{osserman2011lifting} the idea is to approach $\overline X\cdot\overline X'$ (resp. $\mathrm{trop}(X)\cdot\mathrm{trop}(X')$) by $\overline X\cdot z\overline X'$ (resp. $\mathrm{trop}(X)\cdot \mathrm{trop}(zX')$) where $z\in T_N$ such that $\mathrm{trop}(zX')=\mathrm{val}(z)+\mathrm{trop}(X')$ is a small perturbation of $\mathrm{trop}(X')$.

Consider the action $\mu\colon G_m\times X(\Delta)\rightarrow X(\Delta)$ of $G_m$ on $X(\Delta)$ given by $\mu(t,x)=v(t)\cdot x$. Let $p_1\colon G_m\times X(\Delta)\rightarrow G_m$ and $p_2\colon G_m\times X(\Delta)\rightarrow X(\Delta)$ be the two projections. Denote $\mathfrak X=(p_1,\mu)(G_m\times \overline X)$ and $\mathfrak X'=G_m\times \overline X'$.

Note $G_m$ is a nonsingular variety of dimension 1, and $Z=G_m\times X(\Delta)$ is smooth over $G_m$ of relative dimension $n$, and $[\mathfrak X]\in A_{k+1}(Z)$ and $[\mathfrak X']\in A_{l+1}(Z)$. Now by Lemma~\ref{family} we have for $t\in G_m$ the equality $i_t^![\mathfrak X]\cdot i_t^![\mathfrak X']=i_t^!([\mathfrak X]\cdot [\mathfrak X'])\in A_0(\mathfrak X_t\cap\mathfrak X_t')$ where $i_t\colon t\hookrightarrow G_m$ is the inclusion. 

Since $\mathfrak X\rightarrow G_m$ and $\mathfrak X'\rightarrow G_m$ are flat morphisms, we know by excess intersection formula that $i_t^![\mathfrak X]=[\mathfrak X_t]$ and $i_t^![\mathfrak X']=[\mathfrak X'_t]$. On the other hand, assume $[\mathfrak X]\cdot [\mathfrak X']=\sum a_i[\alpha_i]$ where $\alpha_i\in A_1(\mathfrak X\cap \mathfrak X')$ are subvarieties. If $\alpha_i\subset Z_{t_0}=X(\Delta)$ for some $t_0\in G_m$, then $(\alpha_i)_t=\alpha_i$ or $\emptyset$, in both cases we have
$i_t^!(\alpha)=0$ by definition.
Moreover since $Z$ is flat over $G_m$, we have that $i_{Z_t}\colon Z_t\hookrightarrow Z$ is a regular embedding of codimension one and $i_t^!(\gamma)=i_{Z_t}^!(\gamma)=Z_t\cdot \gamma$ in $A_*(\gamma_t)$ for $\gamma$ any subscheme of $Z$. Put these all together we get:
\begin{equation}[\mathfrak X_t]\cdot [\mathfrak X'_t]=\sum_{i=1}^{m} a_i[\alpha_i]\cdot [Z_t]~\mathrm{in}~A_0(\mathfrak X_t\cap \mathfrak X'_t)\end{equation}
where $\alpha_i$ are subvarieties of $\mathfrak X\cap \mathfrak X'$ of dimension 1 which intersect properly with $Z_t$ for all $t\in G_m$. 

We next show that $\int_{(\mathfrak X_t\cap \mathfrak X'_t)_{\overline{|\mathcal P'|}}}\mathfrak X_t\cdot \mathfrak X'_t$ is constant for $t$ with valuation in $[-\epsilon,\epsilon]$, where $(\mathfrak X_t\cap \mathfrak X'_t)_{\overline{|\mathcal P'|}}$ is the part of $\mathfrak X_t\cap \mathfrak X'_t$ that tropicalize to $\overline{|\mathcal P'|}$. Again this number is well-defined according to the discussion below Lemma \ref{connected component}. More generally we claim that under the same assumption of $t$, the number $\int_{(\mathfrak X_t\cap \mathfrak X'_t)_{\overline{|\mathcal P'|}}}\mathfrak [\alpha]\cdot [Z_t]$ is constant for subvarieties $\alpha$ of $\mathfrak X\cap \mathfrak X'$ of dimension one
which intersect properly with $Z_t$ for all $t\in G_m$.

Let $\mathcal S_\epsilon =\mathrm{val}^{-1}[-\epsilon,\epsilon]\subset G_m^{\mathrm{an}}$ and $\mathcal U^{\mathcal P'}=\mathrm{trop}^{-1}(\overline {|\mathcal P'|})\subset X(\Delta)^{\mathrm{an}}$. Let $\mathcal Y_\alpha= \alpha^{\mathrm{an}}\cap(\mathcal S_\epsilon\times \mathcal U^{\mathcal P'})$. We have:
$$\mathcal Y_\alpha= \alpha^{\mathrm{an}}\cap(\mathcal S_\epsilon\times X(\Delta)^{\mathrm{an}})\cap 
(\mathcal S_\epsilon\times  \mathcal U^{\mathcal P'})= \alpha^{\mathrm{an}}\cap(\mathcal S_\epsilon\times X(\Delta)^{\mathrm{an}})\cap (\mathrm{trop}\circ p_2)^{-1}(\overline {|\mathcal P'|}^\circ)$$
since $\alpha$ is supported in $\mathfrak X\cap \mathfrak X'$. Thus $\mathcal Y_\alpha$ is a union of connected components of $\alpha^{\mathrm{an}}\cap(\mathcal S_\epsilon\times X(\Delta)^{\mathrm{an}})$, and is 
Zariski closed in $\mathcal S_\epsilon\times X(\Delta)^{\mathrm{an}}$ and hence
proper over $\mathcal S_\epsilon$ by (\cite[Proposition 4.16]{osserman2011lifting}). Denote $\mathcal S_\epsilon\times \mathcal U^{\mathcal P'}$ by $\mathcal Z$, it follows from \cite[Proposition 5.7]{osserman2011lifting} that:$$
\int_{(\mathfrak X_t\cap \mathfrak X'_t)_{\overline{|\mathcal P'|}}}\mathfrak \alpha\cdot Z_t=\sum_{x\in\mathcal Y_\alpha\cap\mathcal Z_t} i(x,\mathcal Y_\alpha\cdot \mathcal Z_t; \mathcal Z).$$

We will use \cite[Proposition 5.8]{osserman2011lifting}. Look at the projection map $f\colon \mathcal S_\epsilon\times \mathcal Z\rightarrow \mathcal S_\epsilon$ and the Zariski closed subspaces $\mathcal X=\mathcal S_\epsilon\times \mathcal Y_\alpha$ and $\mathcal X'=\Delta(\mathcal S_\epsilon)\times \mathcal U^{\mathcal P'}$ where $\Delta(\mathcal S_\epsilon)$ is the diagonal of $ S_\epsilon\times \mathcal S_\epsilon$. It is easy to see that both $\mathcal X$ and $\mathcal X'$ are flat over $\mathcal S_\epsilon$, and $\mathcal Z$ is an analytic domain in $S_\epsilon\times X(\Delta)^{\mathrm{an}}$, hence quasi-smooth. It follows that $f$ is quasi-smooth and so is $S_\epsilon\times \mathcal Z$ since $S_\epsilon$ is quasi-smooth.

To show $\mathcal X\cap\mathcal X'$ is finite over $\mathcal S_\epsilon$, according to the following fibered diagram we have $\mathcal X\cap \mathcal X'\simeq \mathcal Y_\alpha$.
$$\begin{tikzcd} 
\mathcal X\cap\mathcal X'\rar\dar&\mathcal S_\epsilon\times \mathcal Y_\alpha\rar\dar&\mathcal Y_\alpha\dar\\
\Delta(\mathcal S_\epsilon)\times \mathcal U^{\mathcal P'}\rar&\mathcal S_\epsilon\times\mathcal S_\epsilon\times \mathcal U^{\mathcal P'}\rar&\mathcal S_\epsilon\times \mathcal U^{\mathcal P'}.
\end{tikzcd}$$
Hence $\mathcal X\cap\mathcal X'$ is proper over $\mathcal S_\epsilon$. Moreover $(\mathcal X\cap \mathcal X')_t\simeq \mathcal Y_\alpha\cap \mathcal Z_t$ is finite, since $\alpha\cap Z_t$ is finite. It follows that $\mathcal X\cap \mathcal X'$ is finite over $S_\epsilon$. Now by \cite[Proposition 5.8]{osserman2011lifting} we know that $\int_{(\mathfrak X_t\cap \mathfrak X'_t)_{\overline{|\mathcal P'|}}}\mathfrak \alpha\cdot Z_t$ is constant for $t\in\mathrm{val}^{-1}([-\epsilon,\epsilon])$. The theorem follows from \cite[Theorem 6.3]{osserman2011lifting} since $\mathrm{trop}(\mathfrak X_t)$ intersects $\mathrm{trop}(\mathfrak X_t')$ properly for $t\in\mathrm{val}^{-1}([-\epsilon,0)\cup(0,\epsilon])$.
\end{proof}

\begin{ex}\label{selfintersection}
Let $X=X'$ be the plane curves in $T=\mathrm{Spec}K[x^{\pm},y^{\pm}]$ defined by the equation $f(x,y)=ax^n+by+1=0$ with val$(a)=\mathrm{val}(b)=0$. The tropicalization trop$(X)$ equals the union of rays $R_1=\mathbb R_{\geq 0}\cdot e_1$ and $R_2=\mathbb R_{\geq 0}\cdot e_2$ and $R_3=\mathbb R_{\geq 0}\cdot (-e_1-ne_2)$. These rays have multiplicities $1,n$ and $1$ respectively as showed in the diagram below.
$$ 
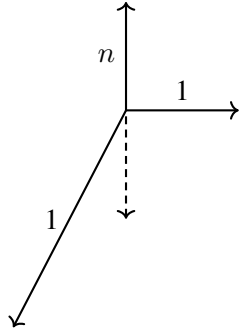
\begin{tikzpicture}[
      scale=1.5,
      level/.style={thick},
      virtual/.style={thick,densely dashed,<-},
      trans/.style={thick,<-},
      classical/.style={thin,double,<->,shorten >=4pt,shorten <=4pt,>=stealth}
    ]
 \draw[trans] (0cm,-5em) -- node[left]{1}  (1cm,0em);
    \draw[virtual] (1cm,-2.5em) -- (1cm,0em) ;
    \draw[trans] (1cm,2.5em) --node[left]{$n$}  (1cm,0em) ; 
      \draw[trans] (2cm,0em) --node[above]{1}  (1cm,0em);
                 \end{tikzpicture}$$

Let $\Delta$ be the complete fan generated by the rays $\rho_1=(1,0),\rho_2=(0,1),\rho_3=(-1,-n),\rho_4=(0,-1)$ and let $D_i, i=1,2,3,4$ be the corresponding torus-invariant divisor. Then $\Delta$ is a smooth compactifying fan for trop$(X)$ and $X(\Delta)$ is a Hirzebruch Surface. For generic choices of $a$ and $b$ we have div$(f)=\overline X-nD_3-D_4$ and hence $[\overline X]=nD_3+D_4$ as cycle classes. According to Theorem \ref{infinite intersection} we have 
$$i_K(\overline X\cdot \overline X',X(\Delta))=i(0,\mathrm{trop}(X)\cdot\mathrm{trop}(X))=n$$
which is the same as the degree of $(nD_3+D_4)\cdot (nD_3+D_4)$.

\end{ex}
\begin{rem}\label{selfinter rem} 
Note that if $\mathrm{trop}(X)\cap\mathrm{trop}(X')$ contains only one component $C$, and $\Delta$ is also compatible with either $\mathrm{trop}(X)$ or $\mathrm{trop}(X')$ and $X(\Delta)$ is a smooth projective variety, as in Example \ref{selfintersection}, then Theorem \ref{infinite intersection} is the same as saying that $\deg([\overline X]\cdot[\overline X'])=\deg([\mathrm{trop}(\overline X)]\ast[\mathrm{trop}(\overline X')])$ (Lemma \ref{chow ring}).  For two reasons: (1) all points of $\overline X\cap\overline X'$ tropicalize to $\overline C$ (cf. \cite[Proposition 3.12]{osserman2011lifting}, take $\mathcal P=\mathrm{trop}(X)$ or $\mathrm{trop}(X')$), and (2) the degree of $[\mathrm{trop}(\overline X)]\ast[\mathrm{trop}(\overline X')]$ is the same as the degree of $\mathrm{trop}( X)\cdot\mathrm{trop}( X')$ (cf. Lemma \ref{intersection compatible}). This is not true in general, see the example below.

\end{rem}

\begin{ex}\label{selfinter ex} 
Let $X,X'$ be curves in $T^2$ defined by $f(x,y)=ax^2+xy+ay^2+x+y+a$ and $g(x,y)=x+y+a$ respectively, where $a\in K$ such that $\mathrm{val}(a)=1$. Then $\mathrm{trop}(X)$ and $\mathrm{trop}(X')$ are as in the following graph, where the red part is the unique component $C$ of $\mathrm{trop}(X)\cap\mathrm{trop}(X')$.
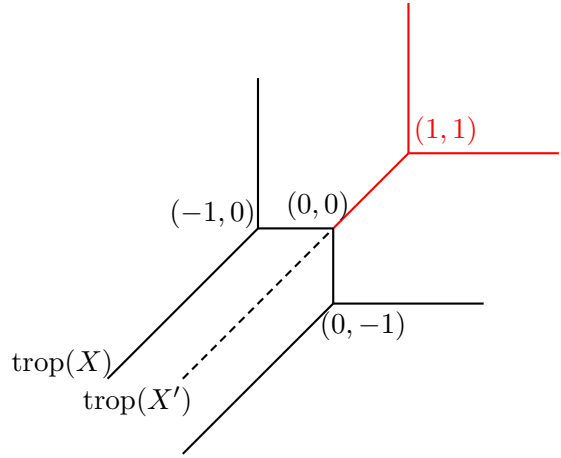
$$ \begin{tikzpicture}[
      scale=1,
      virtual/.style={thick,densely dashed,},
      trans/.style={thick,},
rans/.style={thick,red},
      classical/.style={thin,double,<->,shorten >=4pt,shorten <=4pt,>=stealth}
    ]
 \draw[rans] (1cm,1cm)--(0cm,0cm);
    \draw[virtual] (-2cm,-2cm) -- (0cm,0cm) ;
    \draw[trans] (0cm,0cm) --(-1cm,0cm) ; 
      \draw[trans] (0cm,-1cm) -- (0cm,0cm);
\draw[trans] (-3cm,-2cm) --(-1cm,0cm);
\draw[trans] (-2cm,-3cm) --(0cm,-1cm);
\draw[trans] (2cm,-1cm) --(0cm,-1cm);
\draw[trans] (-1cm,0cm) --(-1cm,2cm) ;
 \draw[rans] (1cm,1cm) --(3cm,1cm);
  \draw[rans] (1cm,1cm) --(1cm,3cm);
 \draw[trans](-0.2cm,0.3cm)node{$(0,0)$};
 \draw[rans](1.5cm,1.3cm)node{$(1,1)$};
 \draw[trans](0.4cm,-1.3cm)node{$(0,-1)$};
 \draw[trans](-1.6cm,0.2cm)node{$(-1,0)$};
 \draw[trans](-3.6cm,-1.8cm)node{$\mathrm{trop}(X)$};
 \draw[trans](-2.6cm,-2.3cm)node{$\mathrm{trop}(X')$};

                 \end{tikzpicture}$$
Let $\Delta$ be the complete fan whose facets are all quadrants of the plane, which is not compatible with $\mathrm{trop}(X)$ or $\mathrm{trop}(X')$. Then $X(\Delta)=\mathbb P^1\times \mathbb P^1$ and $\overline X$ and $\overline X'$ are curves of type $(2,2)$ and $(1,1)$ respectively, and $N_\mathbb R(\Delta)=(\mathbb R\cup \{-\infty,+\infty\})^2$. Straightforward calculation shows that $\deg([\mathrm{trop}(\overline X)]\ast[\mathrm{trop}(\overline X')])=\deg([\overline X]\cdot[\overline X'])=4$ but $\int_{Z_{\overline C}}i^*_{\overline C}(\overline X\cdot \overline X')=\deg(\mathrm{trop}(X)\cdot\mathrm{trop}(X'))=2$. One checks easily that $X\cap X'$ contains two points, counted with multiplicity, that tropicalize to $(1,1)$, while $(\overline X\cap\overline X')\backslash T^2$ contains two points that tropicalize to $(-\infty,-\infty)\not\in \overline C$.
\end{ex}

Using the same definition and argument as in \cite[Theorem 6.10]{osserman2011lifting} we generalize Theorem \ref{infinite intersection} as follows:
\begin{sit}\label{infinite multiple situation}
Let $X_1,...,X_m$ be pure-dimensional closed subschemes of $T$ with $\sum\mathrm{codim}(X_i)=n$ and $m\geq 2$. Let $C$ be a connected component of $\mathrm{trop}(X_1)\cap\cdots\cap\mathrm{trop}(X_m)$ and $\Delta$ a unimodular compactifying fan for $C$.
\end{sit}

\begin{thm}\label{infinite multiple intersection}
In Situation \ref{infinite multiple situation} let $Z_{\overline C}$ be the union of 
irreducible components of $ \cap_i\overline X_i$ which tropicalize to $\overline C$,
and
 $i_{\overline C}\colon Z_{\overline C}\hookrightarrow \cap_i\overline X_i$ be the inclusion. We have: 
$$\int_{Z_{\overline C}}i^*_{\overline C}(\overline X_1\cdots \overline X_m)=\sum_{u\in C} i(u,\mathrm{trop}(X_1)\cdots \mathrm{trop}(X_m))$$where $\overline X_1\cdots \overline X_m$ is the refined intersection product.
\end{thm}

\begin{rem}
Note that Theorem \ref{infinite intersection} (resp. Theorem \ref{infinite multiple intersection}) can be easily generalized to the case where $C$ is a collection of connected components of trop$(X)\cap\mathrm{ trop}(X')$ (resp. $\cap_i\mathrm{trop}(X_i)$). 
\end{rem}

\section{The Higher Dimensional Case}
In this section we restate and prove Theorem \ref{introduction theorem 2}, where $ X\cdot X'$ is a cycle class of possibly positive dimension. For technical reasons we require the ambient space $X(\Delta)$ to be projective, in which case the $(n-1)$-skeleton of $\Delta$ is a tropical hypersurface. This is possible up to a refinement of $\Delta$ due to the toric Chow lemma and the projective resolution of singularities (\cite[Theorem 11.1.9]{cox2011toric}). In this case as a first approach we can consider the degree of $\overline X\cdot \overline X'$ (after restricting to a component of $\mathrm{trop}(X)\cap\mathrm{trop}(X')$) as a cycle class in a projective space, this is indeed equal to the degree of $\mathrm{trop}(X)\cdot\mathrm{trop}(X')$ (after restricting to the same component of $\mathrm{trop}(X)\cap\mathrm{trop}(X')$) intersecting with certain tropical hypersurfaces in $N_\mathbb R$.
Note that in this case $\Delta$ is complete, hence being compatible with a polyhedron is equivalent to being a compactifying fan of that polyhedron. 

\begin{sit}\label{higher situation}
Let $X$ and $X'$ be closed subschemes of $T_N$ of pure dimensions $k$ and $l$ respectively. Let $C$ be a connected component of trop$(X)\cap$trop$(X')$. Let $\Delta$ be a compactifying fan for $C$ such that $X(\Delta)$ is smooth and projective. Let $i_{\overline C}\colon Z_{\overline C}\rightarrow \overline X\cap \overline X'$ be the inclusion of the union of irreducible components of $\overline X\cap \overline X'$ whose tropicalization intersects $\overline C$ (as in Situation \ref{complementary codimension}).
\end{sit}

Given $\Delta$ a compactifying fan for $\mathcal P=C$, we can still find a $\Delta$-thickening $\mathcal P'$ of $\mathcal P$ such that $(\Delta,\mathcal P')$ is a compactifying datum for $X,X'$ and $C$. It follows that we still have 
$$\overline{\mathrm{trop}(X)}\cap\overline{\mathrm{trop}(X')}\cap\overline{|\mathcal P'|}=\overline C\subset \overline{|\mathcal P|}\subset\overline{|\mathcal P'|}^\circ.$$
Hence as in Situation \ref{complementary codimension} $Z_{\overline C}$ is an open and closed subscheme of $\overline X\cap\overline X'$ and the restriction $i^*_{\overline C}(\overline X\cdot\overline X')$ in $A_{k+l-n}(Z_{\overline C})$ is well-defined.

Take a fan $\Delta$ with $X(\Delta)$ smooth and projective. Let $D$ be a torus-invariant ample divisor\footnote{Note that on a smooth complete toric variety ample is equivalent to very ample.} on $X(\Delta)$. We can write $D=\sum_{\rho\in\Sigma(1)}a_\rho D_\rho$. For $\sigma \in \Delta(n)$ we take $m_\sigma\in M$ which satisfies the condition of Proposition \ref{toric divisor} (5), hence $m_\sigma$ is uniquely determined by $D$ and $D$ is also determined by $\{m_\sigma\}_{\sigma\in \Sigma(n)}$. It follows from Proposition \ref{toric divisor}(6)  that $\{m_\sigma\}_{\sigma\in \Delta(n)}$ are all vertices of $P_D$.

Now the set of sections $\{\chi^m|m\in P_D\cap M\}$ gives a closed immersion of $X(\Delta)$ into $\mathbb P^{s-1}$ where $s$ is the order of this set. In order to calculate the degree of $i^*_{\overline C}(\overline X\cdot\overline X')$, note that this is essentially the same as intersecting $k+l-n$ times with a hyperplane in $\mathbb P^{s-1}$, hence with a hyperplane section $L \subset X(\Delta)$ that intersects $T_N$. We can use Theorem \ref{infinite multiple intersection} to get a combinatorial result, as long as that $\Delta$ is a compactifying fan of (components of) $C\cap\mathrm{trop}(L\cap T_N).$

\begin{lem}\label{hyperplane}
Given $\Delta$ and $D$ as above, there exists a hypersurface $H=V(f)\subseteq T_N$ such that $\overline H$ is a hyperplane section of $X(\Delta)$ with respect to $D$ and $\mathrm{trop}(H)=|\Delta(n-1)|$.
\end{lem}
\begin{proof}
Consider the regular functions $$f=\sum\limits_{\sigma\in \Delta(n)} a_\sigma\chi^{m_\sigma}\in K[M]\cap \Gamma(X(\Sigma),\mathcal O(D))
\mathrm{\ where\ val}(a_\sigma)=0 \mathrm{\ for\ all\ } \sigma\in \Delta(n).$$
Let $H=V(f)\subseteq T$. According to Proposition \ref{toric divisor}(6), for generic choice of $\{a_\sigma\}$ we have 
$$0=\mathrm{div}(f)=\overline H-\sum_{\rho\in \Delta(1)}a_\rho D_\rho=\overline H-D\in \mathrm{Cl}(X(\Delta)).$$
Hence $\overline H=D$ as a divisor class. 

On the other hand, since val$(a_\sigma)=0$, the corresponding regular subdivision of Newton polygon $P_D$  of $f$ is trivial. Thus trop$(H)$ is dual to $P_D$ by \cite[Lemma 3.4.6]{maclagan2015introduction}, in other words trop$(H)=|\Delta(n-1)|$. 
\end{proof}

\begin{rem}\label{lattice length}
Note that according to loc.cit. for $\sigma\in \Delta(n-1)$ the multiplicity of $\sigma$ in trop$(H)$ is the lattice length of the edge with end points $m_{\sigma_1}$ and $m_{\sigma_2}$, where $\sigma\subseteq \sigma_1,\sigma_2\in \Delta(n)$. 
\end{rem}


\begin{lem}\label{intersect several ample divisors}
In Situation \ref{higher situation}, for any $T$-invariant ample divisors $D_1,...,D_{k+l-n}$ on $X(\Delta)$ and the corresponding hyperplane sections $\overline H_1,...,\overline H_{k+l-n}$ constructed as above, we have

$$\int_{X(\Delta)}i^*_{\overline C}(\overline X\cdot \overline X')\cdot \overline H_1\cdots \overline H_{k+l-n}=\sum_{u\in C} i(u,\mathrm{trop}(X)\cdot\mathrm{trop}(X')\cdot \mathrm{trop}(H_1)\cdots\mathrm{trop}(H_{k+l-n}))$$
where the product on the left is taken inside $X(\Delta)$, and $i^*_{\overline C}(\overline X\cdot \overline X')$ is considered as a cycle class on $X(\Delta)$ by pushing forward from $Z_{\overline C}$.
\end{lem}
\begin{proof} We first claim that $$\int_{X(\Delta)}i^*_{\overline C}(\overline X\cdot \overline X')\cdot \overline H_1\cdots \overline H_{k+l-n}=\int_{X(\Delta)}i^*_{\overline C}(\overline X\cdot \overline X'\cdot \overline H_1\cdots \overline H_{k+l-n}).$$
Indeed, consider the following diagram:
$$\begin{tikzcd}
Z_{\overline C}\cap \overline H_1\rar\dar{i_{\overline H_1}}& Z_{\overline C}\dar{i_{\overline C}}\\
\overline X\cap\overline X'\cap\overline H_1\rar\dar& \overline X\cap \overline X'\dar\\
\overline H_1\rar{i_1} &X(\Delta)
\end{tikzcd}$$
and note that $i_{\overline C}$ is an open immersion, hence flat, we have 
$$i_{\overline C}^*(\overline X\cdot \overline X')\cdot\overline H_1 =i_1^!(i_{\overline C}^*(\overline X\cdot \overline X'))=i_{\overline H_1}^*(i_1^!(\overline X\cdot \overline X'))=i_{\overline H_1}^*(\overline X\cdot \overline X'\cdot \overline H_1).$$ Since $i_{\overline H_1}$ is also a open immersion and flat, we get the desired equation by induction.

Now note that $Z_{\overline C}\cap(\cap_{i=1}^{k+l-n} \overline H_i)=\{x\in \overline X\cap\overline X'\cap(\cap_{i=1}^{k+l-n}\overline H)|\mathrm{trop}(x)\in \overline C\}=\{x\in \overline X\cap\overline X'\cap(\cap_{i=1}^{k+l-n}\overline H)|\mathrm{trop}(x)\in \overline C\cap(\cap_{i=1}^{k+l-n} \mathrm{trop}(\overline H_i))$ and that $\overline C\cap (\cap_{i=1}^{k+l-n} \mathrm{trop}(\overline H_i))=\overline C\cap (\cap_{i=1}^{k+l-n}\overline{\mathrm{trop}(H_i)})
=\overline{C\cap |\Delta(n-1)|}$ by \cite[Lemma 3.10]{osserman2011lifting}. This is the closure of a union of connected components of $\mathrm{trop}(X)\cap\mathrm{trop}(X')\displaystyle\cap(\cap_{i=1}^{k+l-n}\mathrm{trop}(H_i))$. It is easy to check that $\Delta$ is a compactifying fan for $C\cap |\Delta(n-1)|$. Now the conclusion follows from Theorem \ref{infinite multiple intersection}:
\small{\begin{equation*}
\begin{split}
\int_{X(\Delta)}i^*_{\overline C}(\overline X\cdot \overline X'\cdot \overline H_1\cdots \overline H_{k+l-n})&=\sum\limits_{u\in C\cap \Delta(n-1)} i(u,\mathrm{trop}(X)\cdot\mathrm{trop}(X')\cdot \mathrm{trop}(H_1)\cdots\mathrm{trop}(H_{k+l-n}))\\
&=\ \ \ \ \ \ \sum\limits_{u\in C}\quad\quad i(u,\mathrm{trop}(X)\cdot\mathrm{trop}(X')\cdot \mathrm{trop}(H_1)\cdots\mathrm{trop}(H_{k+l-n})).
\end{split}
\end{equation*}}

\end{proof}

Setting $H_1=H_2=\cdots=H_{k+l-n}$ we have the following corollary, which is an analogue of Theorem \ref{infinite intersection}:

\begin{cor}\label{preserve degree}
In Situation \ref{higher situation}, for any $T$-invariant ample divisor $D$ on $X(\Delta)$, under the embedding $X(\Delta)\rightarrow \mathbb P^{s-1}$ induced by $D$ we have
$$\deg(i^*_{\overline C}(\overline X\cdot \overline X'))=\sum_{u\in C} i(u,\mathrm{trop}(X)\cdot\mathrm{trop}(X')\cdot (\Delta(n-1))^{k+l-n})$$where the multiplicities of facets of $\Delta(n-1)$ is given as in remark \ref{lattice length}.
\end{cor}
\begin{proof} Let $H$ be the corresponding hypersurface in $T_N$ as before. Let $\widetilde H$ be a hyperplane in $\mathbb P^{s-1}$ such that $\widetilde H\cap X(\Delta)=\overline H$ (the closure of $H$ in $X(\Delta)$).
Then intersecting with $\widetilde H$ in $\mathbb P^{s-1}$ is essentially the same as intersecting with $\overline H$ in $X(\Delta)$: consider the fiber product: 
$$\begin{tikzcd}  \overline H\rar{i_{\overline H}}\dar{} &X(\Delta)\dar{}\\ \widetilde H\rar{i_{\widetilde H}} &\mathbb P^{s-1},
\end{tikzcd}$$
for a closed subvariety $\alpha$ of $X(\Delta)$ we have $\alpha\cdot \widetilde H=i_{\widetilde H}^!([\alpha])=i_{\overline H}^!([\alpha])=\alpha\cdot \overline H$ where the first intersection product is in $\mathbb P^{s-1}$ while the last one is in $X(\Delta)$. 
Hence the corollary follows from Lemma \ref{intersect several ample divisors}.
\end{proof}

The following lemma is a complement of Theorem \ref{chow group}:

\begin{lem}\label{chow ring}
If $X(\Delta)$ is smooth and projective then the isomorphism $$\phi\circ\varphi\colon A_*(X(\Delta))\rightarrow A_*(N_\mathbb R(\Delta))$$ in Theorem \ref{chow group} is induced by the tropicalization map.\footnote{Note that a more general case, where $\Delta$ is only assumed to be complete unimodular, is considered in \cite[Theorem 3.21]{meyer2011intersection}. However, it appears that the argument in its proof is using (without proof) that $\deg(\mathrm{trop}(X)\cdot\mathrm{trop}(X'))=\deg([\mathrm{trop}(X)]\ast[\mathrm{trop}(X')])$ when $\mathrm{trop}(X)$ and $\mathrm{trop}(X')$ are of complementary codimension and intersect transversally and the intersection $\overline{\mathrm{trop}(X)}\cap\overline{\mathrm{trop}(X')}$ is supported on $N_\mathbb R$. This is not obvious to the author, hence we sketch the proof in our special situation for completeness.}
\end{lem}

For $[\alpha]\in A_*(X(\Delta))$ by $\mathrm{trop}([\alpha])$ we mean the tropicalization of any representative of $\alpha$ as a sum of irreducible closed subschemes of $X(\Delta)$. More specifically, let $\beta$ be an irreducible closed subscheme of $X(\Delta)$ and $\mathcal V(\tau)\cong X(\mathrm{Star}_\Delta(\tau))$ be the maximal closed torus orbit that contains $\beta$ as a closed subset. Write $[\beta]=a[\beta']$ where $\beta'$ is reduced irreducible, hence a closed subscheme of $\mathcal V(\tau)$. We then define $\mathrm{trop}([\beta])={\mathrm{trop}(\beta'}) $ with multiplicities equal to $a$ times the multiplicities of $\mathrm{trop}(\beta'\cap\mathcal O(\tau))$. 
Note that although our notation is non-standard, the following theorem/corollary does not depend on the choice of representative.

\begin{proof}[Proof of Lemma \ref{chow ring}]
Let $\alpha$ be an reduced irreducible closed subscheme of $X(\Delta)$, we need to show that $\varphi([\alpha])=\phi^{-1}([\mathrm{trop}(\alpha)])$. First assume $\alpha=D_{\rho_\alpha}$ is a torus-invariant divisor, hence $\mathrm{trop}(\alpha)=V(\rho_\alpha)$. We have $\varphi([\alpha])=m_\alpha$ where $m_\alpha(\tau)=\deg([\alpha]\cdot[\mathcal V(\tau)])$ for all $\tau\in\Delta(n-1)$. On the other hand let $f$ be the support function on $\Delta$ defined by $f|_\sigma=\langle\ \cdot\ ,m_\sigma\rangle$ for $\sigma\in\Delta(n)$ and $m_\sigma\in M$ such that $f(u_\rho)=-1$ if $\rho=\rho_\alpha$ and $f(u_\rho)=0$ otherwise, where $u_\rho$ is the lattice generator of $\rho\in\Delta(1)$. Then $f\cdot N_\mathbb R(\Delta)=V(\rho_\alpha)-\overline F$ where $F$ is a subfan of $\Delta$ of codimension one with weights 
$$m'_\alpha(\tau)=\left\{\begin{array}{ccc}0&\mathrm{if}\ \rho_\alpha+\tau\not\in\Delta(n),\\1&\mathrm{if}\ \rho_\alpha+\tau\in\Delta(n),\\-\langle u_{\rho_1},m_{\sigma_2}\rangle&\mathrm{if}\ \rho_\alpha\in\tau.\end{array}\right.$$
where $\sigma_1$ and $\sigma_2$ are the two maximal faces of $\Delta$ that contain $\tau$ and $\rho_i\in\sigma_i(1)\backslash\tau(1)$. It then follows that $\phi^{-1}([\mathrm{trop}(\alpha)])=m'_\alpha$ and one checks directly that $m_\alpha=m'_\alpha$.

Next let $\alpha=\overline H$ be as in Lemma \ref{hyperplane}. We have $[\alpha]=\sum a_\rho D_\rho$ and $[\mathrm{trop}(\alpha)]=\sum a_\rho[V(\rho)]$ (by Remark \ref{lattice length} and looking at the support function defined by $\{m_\sigma\}$ corresponding to $D$ in Lemma \ref{hyperplane}). Thus $\varphi([\alpha])=\phi^{-1}([\mathrm{trop}(\alpha)])$. The same is true for $\alpha=g\overline H$ for any $g\in T_N$.

Now let $\alpha$ be an arbitrary irreducible closed subscheme of $X(\Delta)$ of dimension $k$ that intersects $T_N$. Given ample divisors $\overline H_1,...,\overline H_{n-k}$ as in Lemma \ref{hyperplane}. Take  $g_1,...,g_{n-k}\in T_N$ such that $\mathrm{trop}(\alpha),\mathrm{trop}(g_1H_1)
,$
$...,\mathrm{trop}(g_{n-k}H_{n-k})$ intersect properly, and $\mathrm{trop}( \alpha)\cap\mathrm{trop}(g_1{\overline H_1})\cap\cdots\cap{ \mathrm{trop}(g_{n-k}{\overline H_{n-k}})}={\mathrm{trop}(\alpha)}\cap\mathrm{trop}(g_1{H_1})\cap\cdots\cap \mathrm{trop}({g_{n-k}H_{n-k}})$ (\cite[Lemma 3.10]{osserman2011lifting}). It follows from \cite[Theorem 5.2.3]{osserman2013lifting} that
$$\deg([\alpha]\cdot[ g_1\overline H_1]\cdots[ g_{n-k}\overline H_{n-k}])=\deg(\mathrm{trop}(\alpha\cap T_N)\cdot\mathrm{trop}(g_iH_1)\cdots\mathrm{trop}(g_{n-k}H_{n-k})).$$
By Lemma \ref{intersection compatible} and the fact that $\mathrm{trop}(g_iH_i)$ is rationally equivalent to $\mathrm{trop}(H_i)$, the equation above is equivalent to:
$$\deg([\alpha]\cdot[\overline H_1]\cdots[\overline H_{n-k}])=\deg([\mathrm{trop}(\alpha)]\ast[\mathrm{trop}(\overline H_1)]\ast\cdots\ast[\mathrm{trop}(\overline H_{n-k})]).$$
Since $A_*(X(\Delta))$ is generated by divisors of form $\overline H$ as in Lemma \ref{hyperplane} and $\phi\circ\varphi$ in Theorem \ref{chow group} is a ring isomorphism, this implies $\varphi([\alpha])=\phi^{-1}([\mathrm{trop}(\alpha)])$.

For arbitrary $\alpha$ let $\mathcal V(\tau)$ be the maximal closed orbit that contains $\alpha$. We can write $$[\alpha]=\sum\limits_{\tau\prec\sigma\in\Delta(n-k)}a_\sigma[\mathcal V(\sigma)].$$ The discussion above shows that $[\mathrm{trop}(\alpha)]=\sum a_\sigma [V(\sigma)]$ as tropical cycle classes on $V(\tau)$, hence they are in the same class of cycles on $N_\mathbb R(\Delta)$. It follows that $\varphi([\alpha])=\phi^{-1}([\mathrm{trop}(\alpha)])$.
\end{proof}

We now state the main theorem of this section, which comes as a corollary of Lemma \ref{intersect several ample divisors} when we let $H_i$ vary. Let $\mathrm{trop}(X)\cdot\mathrm{trop}(X')|_{ C}$ be the part of $\mathrm{trop}(X)\cdot\mathrm{trop}(X')$ that is supported on $C$, which is a tropical cycle in $N_\mathbb R$. 

\begin{thm}\label{intersect rationally equivalent}
With the same notation as in Lemma \ref{intersect several ample divisors} we have that $[\mathrm{trop}(i^*_{\overline C}(\overline X\cdot \overline X'))]=[\overline{\mathrm{trop}(X)\cdot\mathrm{trop}(X')|_{ C}}]$ as cycle classes in $A_{k+l-n}(N_\mathbb R(\Delta))$. 
\end{thm}
\begin{proof}
According to Lemma \ref{chow ring} and Lemma \ref{intersect several ample divisors} and Lemma \ref{intersection compatible} we have:
\begin{equation*}\begin{split}\deg([\mathrm{trop}(i^*_{\overline C}(\overline X\cdot \overline X'))]\ast [\mathrm{trop}(\overline H_1)]\ast\cdots\ast[\mathrm{trop}(\overline H_{k+l-n})])\ \ \ \ \ \ \ \\
=\deg([\overline{\mathrm{trop}(X)\cdot\mathrm{trop}(X')|_C}]\ast [\mathrm{trop}(\overline H_1)]\ast\cdots\ast[\mathrm{trop}(\overline H_{k+l-n})]).
\end{split}
\end{equation*}
Since the cycles of the form $\mathrm{trop}(\overline H)$ generate $A_*(N_\mathbb R(\Delta))$ where $H$ is defined as in Lemma \ref{hyperplane}, we got the desired conclusion. 
\end{proof}

\begin{rem}
When $\mathrm{trop}(X)$ and $\mathrm{trop}(X')$ intersect properly taking $\Delta=\{0\}$ is enough. It follows from \cite[Corollary 5.1.2]{osserman2013lifting} that we actually have $\mathrm{trop}(i^*_{ C}( X\cdot  X'))=\mathrm{trop}(X)\cdot\mathrm{trop}(X')|_{ C}$ as tropical cycles in $N_\mathbb R$.
\end{rem}

\begin{cor}\label{compact stable equality}
With the same notation as in Lemma \ref{intersect several ample divisors} we have:
$$\deg(\mathrm{trop}(i^*_{\overline C}(\overline X\cdot\overline X'))\cdot _c \overline F))=\deg(\overline{\mathrm{trop}(X)\cdot\mathrm{trop}(X')|_C}\cdot_c \overline F),$$
or equivalently 
$$\deg(\mathrm{trop}(i^*_{\overline C}(\overline X\cdot\overline X'))\cdot _c \overline F))=\deg(\mathrm{trop}(X)\cdot\mathrm{trop}(X')|_C\cdot  F)$$
for all representatives of $i^*_{\overline C}(\overline X\cdot\overline X')$ and $F\subset N_\mathbb R$ such that $F$ is compatible with $\Delta$.
\end{cor}
Note that the corollary is not necessarily true if we ignore the compatibility condition for $F$ and use the stable intersection on $N_\mathbb R$ for the left hand side of the equation.

\begin{ex}
Let $X=X'$ be planes in $T^3$ defined by the equation $x+y+z+1=0$. Then $\mathrm{trop}(X)=\mathrm{trop}(X')$ is the standard tropical plane $P$ in $\mathbb R^3=\mathbb R e_x+\mathbb R e_y+\mathbb R e_z$ where $e_x,e_y,e_z\in N$ are dual to $x,y,z\in M$. The intersection of the tropicalizations has only one component $P$ and the stable intersection $\mathrm{trop}(X)\cdot\mathrm{trop}(X')$ is the standard tropical line in $\mathbb R^3$. Let $\Delta$ be the complete fan whose 2-skeleton agrees with $P$, then $X(\Delta)=\mathbb P^3$ and $\overline X\cdot\overline X'$ is a projective line in $\overline X$. First let $F$ be the plane $\mathbb R e_y+\mathbb R e_z$. Take a representative of $\alpha=i^*_{\overline P}(\overline X\cdot\overline X')$ by the line $x=a, y+z+1-a=0$ with $\mathrm{val}(a)=\mathrm{val}(1-a)=0$. Then $\mathrm{trop}(\alpha\cap T^3)$ is the standard tropical line in $F$. We have $$\deg(\mathrm{trop}(\alpha\cap T^3)\cdot F))=0\neq 1=\deg(\mathrm{trop}(X)\cdot\mathrm{trop}(X')|_P\cdot F).$$
On the other hand If we take $F=P$ then one checks easily that both sides of the equation in Corollary \ref{compact stable equality} are equal to 1.
\end{ex}


\section{Lifting within Nontoric Ambient Spaces}
In this section we assume $X$ and $X'$ are of complementary codimension in a subscheme $Y$ of an algebraic torus, which is not necessarily smooth. Even if $Y$ is smooth, there is no obvious condition for $\Delta$ such that the closure $\overline Y$ in $X(\Delta)$ is still smooth. In this case the intersection cycle $\overline X\cdot\overline X'$ in $\overline Y$ is not well-defined, however the intersection multiplicity is still valid at isolated points of $\overline X\cap\overline X'$ at which $\overline Y$ is regular, and indeed it is possible to choose $\Delta$ such that $\overline Y$ is at least smooth at points of $\overline X\cap\overline X'$ that we are interested in.

\begin{sit}\label{reduced ambient space}
Assume $Y$ is a reduced closed subscheme of $T_N$ of pure dimension $d$. Let $X$ and $X'$ be closed subschemes of $Y$ of pure dimensions $k$ and $l$ such that $k+l=d$. Let $C$ be a connected component of $\mathrm{trop}(X)\cap\mathrm{trop}(X')$ that is contained in the relative interior of a maximal face $\iota$ of $\mathrm{trop}(Y)$ of multiplicity one. Let $W_\mathbb R$ be the affine subspace in $N_\mathbb R$ of dimension $d$ parallel to $\iota$ and $W=W_\mathbb R\cap N$ the corresponding sublattice. Let $\Delta$ be a unimodular fan contained in $W_\mathbb R$ which is a compactifying fan for $\mathrm{trop}(Z)\cap\mathrm{trop}(Z')\cap C$, where $Z$ and $Z'$ run over all irreducible components of $X$ and $X'$ respectively. Assume there are finitely many points of $\overline X\cap\overline X'$ which tropicalize to $\overline C$, in other words $Z_{\overline C}$ is a finite set.
\end{sit}

We first check that in the above situation $\overline Y$ is smooth at points which tropicalize to $\overline C$.
Let $X(\Delta)$ be as before and $Y(\Delta)$ the toric variety associated to $\Delta$ where $\Delta$ is considered as a fan in $W_\mathbb R$. Note that $X(\Delta)=Y(\Delta)\times T^{n-d}$ and $N_\mathbb R(\Delta)=W_\mathbb R(\Delta)\times \mathbb R^{n-d}$. We denote by $\pi\colon X(\Delta)\rightarrow Y(\Delta)$ and $\widetilde\pi\colon N_\mathbb R(\Delta)\rightarrow W_\mathbb R(\Delta)$ the projections, which are induced by a splitting of $N\rightarrow N/W$, and $\mathrm{trop}_Y\colon Y(\Delta)\rightarrow W_\mathbb R(\Delta)$ the tropicalization map on $Y(\Delta)$. Take $P\subset \mathrm{relint}(\iota)$ an integral $G$-affine polyhedron such that $\rho(P)\in \Delta$ is a face. Let $P'$ be the image of $P$ in $W_\mathbb R$ under the projection $\widetilde\pi$. 

As in \cite[$\S$4.29]{baker2011nonarchimedean} we let $\mathcal U^P$ be the preimage of $\overline P$ under the map $\mathrm{trop}\colon X(\Delta)^{\mathrm{an}}\rightarrow N_\mathbb R(\Delta)$. This is the same as the preimage under $\mathrm{trop}\colon X(\rho(P))^{\mathrm{an}}\rightarrow N_\mathbb R(\rho(P))$, hence according to \cite[Proposition 6.9]{rabinoff2012tropical} $\mathcal U^P$ is an affinoid domain in $X(\Delta)^{\mathrm{an}}$ whose ring of global sections is integral. Similarly we define $\mathcal U^{P'}$ to be the preimage of $\overline {P'}$ under $\mathrm{trop}_Y\colon Y(\Delta)^{\mathrm{an}}\rightarrow W_\mathbb R(\Delta)$, which is an affinoid domain in $Y(\Delta)^{\mathrm{an}}$. Let $\mathcal Y^P=\overline Y^{\mathrm{an}}\cap \mathcal U^P$ and $\pi_P\colon \mathcal Y^P\rightarrow \mathcal U^{P'}$ the map induced by the projection. 

\begin{lem}\label{projection isomorphism}
$\pi_P$ is an isomorphism.
\end{lem}
\begin{proof}
Assume $\mathcal Y^P=\mathcal M(\mathcal B)$ and $\mathcal U^{P'}=\mathcal M(\mathcal A)$ where $\mathcal A$ is integral. According to Theorem 4.30 and Corollary 4.32 of \cite{baker2011nonarchimedean} $\pi_P$ is a finite morphism of pure degree one in the sense of \cite[\S 3.15]{baker2011nonarchimedean}.\footnote{In \cite[Theorem 4.30]{baker2011nonarchimedean} the conclusion is proved for $\Delta=\{0\}$ and $P$ a polytope which has trivial recession cone, their argument still works in our case where $\Delta=\rho(P)$.
} Look at the induced map $\phi\colon \mathrm{Spec}(\mathcal B)\rightarrow \mathrm{Spec}(\mathcal A)$, which is also finite of pure degree one, so $\mathrm{Spec}(\mathcal B)$ is irreducible. We first claim that $\mathrm{Spec}(\mathcal B)$ is integral. Indeed, since $\overline Y$ is reduced, we know $\overline Y^\mathrm{an}$ is reduced, hence so is $\mathcal Y^P$ as an affinoid domain in $\overline Y^\mathrm{an}$, therefore $\mathrm{Spec}(\mathcal B)$ is reduced.


On the other hand, since $Y(\Delta)^{\mathrm{an}}$ is quasi-smooth, so is $\mathcal U^{P'}$, in particular $\mathrm{Spec}(\mathcal A)$ is smooth, and hence $\mathcal A$ is an integral normal ring. Now we have a finite morphism of degree one between integral domains whose source is normal, it must be an isomorphism. Therefore, $\pi_P$ is an isomorphism.
\end{proof}

Note that Lemma \ref{projection isomorphism} generalizes easily to the case where $P=|\mathcal P|$ for a finite collection $\mathcal P$ of integral $G$-affine polyhedra in $\mathrm{relint}(\iota)$ of which $\Delta$ is a compactifying fan. 
In this case $\mathcal Y^P$ as an analytic domain in $\overline Y^{\mathrm{an}}$ is quasismooth, hence $\overline Y^{\mathrm {an}}$ is regular at points tropicalize to $\overline P$, therefore $\overline Y$ is regular at points tropicalize to $\overline P$. We now have a well-defined intersection multiplicity $i(x,\overline X\cdot\overline X',\overline Y)$ for $x$ such that $\mathrm{trop}(x)\in\overline C$, and are able to state the theorem for non-toric ambient space:

\begin{thm}\label{lifting reduced ambient space}
In situation \ref{reduced ambient space} we have 
$$\sum_{x\in Z_{\overline C}} i(x,\overline X\cdot\overline X';\overline Y)=\sum_{u\in C} i(u,\mathrm{trop}(X)\cdot\mathrm{trop}(X');\mathrm{trop}(Y)).$$
\end{thm}
\begin{proof}
We show by passing to analytic spaces that in the above equality, we can replace $Y$ with $T_W$ and replace $X$ and $X'$ with subschemes of $T_W$.
Since intersection multiplicities and tropicalizations are additive with respect to cycles, we may assume both $X$ and $X'$ are reduced. First assume $\dim (C)\leq\min\{k-1,l-1\}$. Choose a splitting of $N\rightarrow N/W$ such that the corresponding projection yields $$\widetilde \pi^{-1}(\widetilde\pi(C))\cap \mathrm{trop}(X)=\widetilde \pi^{-1}(\widetilde\pi(C))\cap \mathrm{trop}(X')=C.$$
Take also a $\Delta$-thickening $\mathcal P$ of $C$ such that $P=|\mathcal P|\cap \mathrm{trop}(Y)$ is contained in $\mathrm{relint}(\iota)$ and that $$\widetilde \pi^{-1}(\widetilde\pi(P))\cap \mathrm{trop}(X)=
\widetilde \pi^{-1}(\widetilde\pi(P))\cap \mathrm{trop}(X)\cap\iota$$$$ 
\widetilde \pi^{-1}(\widetilde\pi(P))\cap \mathrm{trop}(X')=
\widetilde \pi^{-1}(\widetilde\pi(P))\cap \mathrm{trop}(X')\cap\iota.$$
Let $\widetilde X$ and $\widetilde X'$ be the scheme theoretic image of $X$ and $X'$ under $\pi$ in $Y(\Delta)$, which can be identified as the closure of $\pi(X)$ and $\pi(X')$ in $Y(\Delta)$ where $\pi(X)$(resp. $\pi(X')$) is the scheme theoretic image of $X$ (resp. $X'$) in $T_W$. 
Let $\mathcal X^P=\overline X^{\mathrm{an}}\cap\mathcal U^P$ and $\mathcal X'^P=\overline X'^{\mathrm{an}}\cap \mathcal U^P$, let $\widetilde{\mathcal X}^{P'}=\widetilde X^{\mathrm{an}}\cap\mathcal U^{P'}$ and $\widetilde{\mathcal X}'^{P'}=\widetilde X'^{\mathrm{an}}\cap\mathcal U^{P'}$. We then have $\varphi\colon \mathcal X^P\rightarrow\widetilde{\mathcal X}^{P'}$ induced by the projection. To show that $\varphi$ is an isomorphism. 

We can assume $P$ is a polyhedron and $\rho(P)\in \Delta$.
First check that $\varphi $ is surjective. Since $X\xrightarrow{\pi} \widetilde X$ is dominant, the induced $ \overline X^{\mathrm{an}}\xrightarrow{\pi^{\mathrm{an}}} \widetilde X^{\mathrm{an}}$ is also dominant. There exists a $\Delta$-thickening $\mathcal Q'\subset \mathrm{relint}(\iota')$ of $P'$, such that 
$\widetilde \pi^{-1}(|\mathcal Q'|)\cap \mathrm{trop}(X)=|\mathcal Q|\cap\mathrm{trop}(X)$, where $\mathcal Q$ is the preimage under $\widetilde \pi$ of $\mathcal Q'$ in $\iota$.
 It follows that the map
$$
\overline X^{\mathrm{an}}\cap \mathrm{trop}^{-1}(\widetilde\pi^{-1}(\overline{|\mathcal Q'|}^\circ)\cap\overline{|\mathcal Q|})
=\overline X^{\mathrm{an}}\cap \mathrm{trop}^{-1}(\widetilde\pi^{-1}(\overline{|\mathcal Q'|}^\circ))
\xrightarrow{\pi^{\mathrm{an}}} \widetilde X'^{\mathrm{an}}\cap\mathrm{trop}_Y^{-1}(\overline{|\mathcal Q'|}^\circ)$$
is dominant. As $\pi_{\mathcal Q}$ is isomorphism, the image of the map above is closed, hence the map is surjective, it follows that $\varphi$ is surjective. On the other hand since $\pi_P$ is isomorphism, $\varphi$ is a closed immersion, also the same argument as in Lemma \ref{projection isomorphism} shows that both $\mathcal X^P$ and $\widetilde{\mathcal X}^{P'}$ are reduced, so $\varphi$ is an isomorphism.

Note that the same argument shows that $\mathcal X'^P$ is isomorphic to $\widetilde{\mathcal X}'^{P'}$ under the projection. Now for every point $x\in X_{\overline{C}}$ we have 
\begin{equation*}
\begin{split}
i(x,\overline X\cdot\overline X';\overline Y)&=i(x,\overline X^{\mathrm{an}}\cdot\overline X'^{\mathrm{an}};\overline Y^{\mathrm{an}})=i(x,\mathcal X^P\cdot \mathcal X'^P;\mathcal Y^P)\\&=i(\pi^{\mathrm{an}}(x),\widetilde{\mathcal X}^{P'}\cdot \widetilde{\mathcal X}'^{P'};\mathcal U^{P'})=i(\pi(x),\widetilde{ X}\cdot \widetilde{ X}';Y(\Delta))\\
&=i(\pi(x),\overline{\pi(X)}\cdot\overline{\pi(X')};Y(\Delta))
\end{split}
\end{equation*}
by  \cite[Proposition 5.7]{osserman2011lifting}. We next check that $\mathrm{trop}(X)$ and $\mathrm{trop}_Y(\pi(X))$ have the same multiplicities at faces contained in $P$ and $P'$ respectively.

Let $\sigma\subset P$ be a bounded face of $\mathrm{trop}(X)$ such that $\widetilde\pi(\sigma)$ is also a face of $\mathrm{trop}_Y(\pi(X))$. Let $L$ be the sublattice of $W$ of rank $k$ such that $L_\mathbb R$ is parallel to $\sigma$. Let $\pi_W\colon T_W\rightarrow T_L$ be the projection corresponds to $\widetilde\pi_W\colon W_\mathbb R\rightarrow L_\mathbb R$. Let $\pi_N=\pi_W\circ\pi\colon T_N\rightarrow T_L$ and $\widetilde{\pi}_N=\widetilde\pi_W\circ\widetilde\pi\colon N_\mathbb R\rightarrow L_\mathbb R$. For any $\bold w\in\mathrm{relint}(\sigma)$ we have:

$$\begin{tikzcd}  X^\mathrm{an}\cap\mathcal U^\bold w\rar{\varphi}\drar{\pi_N^\mathrm{an}} &{\pi(X)}^{\mathrm{an}}\cap\mathcal U^{\widetilde\pi(\bold w)}\dar{\pi_W^{\mathrm{an}}}\\  &\mathcal U^{\widetilde{\pi}_N(\bold w)}.
\end{tikzcd}$$

Since $\varphi$ is an isomorphism, we have $\deg(\pi_N^\mathrm{an})=\deg(\pi_W^\mathrm{an})$(when restricted to the above diagram). By \cite[Corollary 4.32]{baker2011nonarchimedean} we have $$m_{\mathrm{trop}(X)}(\sigma)=m_{\mathrm{trop}_Y(\pi(X))}(\widetilde\pi(\sigma)).$$
Now we can replace $Y$ with $T_W$ and replace $X$ and $X'$ with $\pi(X)$ and $\pi(X')$, then the theorem follows from \cite[Theorem 6.4]{osserman2011lifting}.

For the general case we take $N_+=N\oplus \mathbb Z\oplus \mathbb Z$, hence $T_{N_+}=T_N\times G_\mathrm{m}^2$. Take also $Y_+=Y\times G_\mathrm{m}^2$ and $X_+=X\times G_\mathrm{m}\times \{1\}$ and 
$X'_+=X'\times \{1\}\times G_\mathrm{m}$. Then $\mathrm{trop}(Y_+)=\mathrm{trop}(Y)\times \mathbb R^2$ and 
$\mathrm{trop}(X_+)=\mathrm{trop}(X)\times \mathbb R\times\{0\}$ and 
$\mathrm{trop}(X'_+)=\mathrm{trop}(X')\times\{0\}\times \mathbb R$ as tropical cycles,
and $\mathrm{trop}(X_+)\cap\mathrm{trop}(X'_+)=\mathrm{trop}(X)\cap\mathrm{trop}(X')$. Let $X(\Delta)_+$ be the toric variety associated to $\Delta$ as a fan in $(N_+)_\mathbb R$. We also have $\overline Y_+=\overline Y\times G_{\mathrm m}^2$ and 
$\overline X_+=\overline X\times G_{\mathrm m}\times\{1\}$ and $\overline X'_+=\overline X'\times\{1\}\times G_{\mathrm m}$ and $\overline X_+\cap\overline X'_+=\overline X\cap\overline X'$. It follows from the projection formula that for all isolated points $x\in\overline X_+\cap\overline X_+'$ we have: 

$$i(x,\overline X_+\cdot \overline X'_+;\overline Y_+)=i(x,\overline X_+\cdot (\overline X'\times \{1\});\overline Y\times G_{\mathrm m})=i(x,\overline X\cdot \overline X';\overline Y).$$
Also for $u\in C$ we have 
\begin{equation*}
\begin{split}
i(u,\mathrm{trop} (X_+)\cdot \mathrm{trop}(X'_+);\mathrm{trop}(Y_+))&=i(u,\mathrm{trop} (X_+)\cdot (\mathrm{trop}(X')\times\{0\});\mathrm{trop}(Y)\times\mathbb R)\\&=i(u,\mathrm{trop}(X)\cdot\mathrm{trop}(X');\mathrm{trop}(Y)).
\end{split}
\end{equation*}
Note that $\dim C\leq\dim(X)=\dim(X_+)-1$ and $\dim C\leq\dim(X')=\dim(X'_+)-1$, hence we reduced to the  case above and the theorem is proved.
\end{proof}

\bibliographystyle{amsalpha}\bibliography{1}
\end{document}